\documentclass{article}
\usepackage{amsmath}
\usepackage{physics}
\usepackage{amsfonts}
\usepackage{bbm}
\usepackage{enumitem}
\usepackage[utf8]{inputenc}
\usepackage{graphicx}
\usepackage[space]{grffile}
\usepackage[english]{babel}
\usepackage{verbatim}
\usepackage{amsthm}
\usepackage{fancyhdr}
\usepackage{hyperref}
\hypersetup{colorlinks=true, linkcolor=blue, citecolor=blue}
\usepackage[capitalize]{cleveref}
\usepackage{fancybox}
\usepackage{mdframed}
\usepackage{mathtools}
\usepackage{dirtytalk}
\usepackage[ruled, vlined,  linesnumbered]{algorithm2e}

\allowdisplaybreaks

\newtheorem{thm}{Theorem}[section]
\newtheorem*{thmnn}{Theorem}
\newtheorem{lemma}[thm]{Lemma}
\newtheorem*{lemmann}{Lemma}
\newtheorem{prop}[thm]{Proposition}

\newtheorem*{notn}{Notation}

\theoremstyle{definition}
\newtheorem{defn}{Definition}[section]

\theoremstyle{remark}
\newtheorem*{rem}{Remark}

\newcommand{\Title}{ Optimal Recovery of  Block Models with $q$ Communities}
\newcommand{\Class}{}
\newcommand{\Name}{By Byron Chin, Allan Sly}

\newcommand{\Var}[1]{\text{Var}\left( #1 \right)}
\newcommand{\E}[1]{\mathbb{E}\left[ #1 \right]}
\newcommand{\Prob}[1]{\mathbb{P}\left( #1 \right)}

\title{\vspace{-1.4cm} \textbf \Title \vspace{-0.5cm}}
\author{\Name \vspace{0.2cm}}
\date{}

\headwidth \textwidth
\begin{document}
	
\maketitle
\thispagestyle{plain}
\vspace*{-1.2cm}
\begin{center}
\begin{minipage}{0.75\textwidth}
		\paragraph{} This paper is motivated by the reconstruction problem on the sparse stochastic block model. The paper "Belief propagation, robust reconstruction and optimal recovery of block models" by Mossel, Neeman, and Sly \cite{beliefpropogation} provided and proved a reconstruction algorithm that recovers an optimal fraction of the communities in the 2 community case. The main step in their proof was to show that when the signal to noise ratio is sufficiently large, in particular $\theta^2d > C$, the reconstruction accuracy on a regular tree with or without noise on the leaves is the same. This paper will generalize their results, including the main step, to any number of communities, providing an algorithm related to Belief Propagation that recovers a provably optimal fraction of community labels. 
\end{minipage}
\end{center}

\section{Introduction}
\subsection{Background}
The problem of community detection is one of great relevance in machine learning and data science. It is central to the discovery of interesting structure within networks, which can be anything from social media networks to biochemical networks. While there are many effective algorithms in practice, such as Belief Propagation, the theory behind these algorithms is not as well understood. The Stochastic Block Model is one of the most well-studied random graph models for community detection. The set-up of the model is to have some fixed number of vertices $n$ partitioned into $q$ communities. The edges of the graph are drawn between these vertices independently at random, with probabilities according to a symmetric matrix $M$, where $M_{ij}$ represents the probability that a vertex in community $i$ and community $j$ are joined by an edge. From this general definition, the simplest model is generated when the $q$ communities are of approximate equal size $\frac{n}{q}$, and the matrix $M$ is of the form $M_{ii} = p$, $M_{ij} = q$. In particular, there is a uniform probability for drawing edges between vertices of the same community, and another uniform probability for drawing edges between vertices of different communities. In the setting that motivates this paper, we moreover assume that $p = \frac{a}{n}$ and $q = \frac{b}{n}$ for some constants $a$ and $b$. This is known as the sparse stochastic block model, as no matter the value of $n$ the expected degree of a given vertex is constant. Another possibility is that the average degree is logarithmic, which occurs when $p$ and $q$ have the form $\frac{a\log n}{n}$ and $\frac{b\log n}{n}$. It is a known result in this logarithmic regime that exact recovery of the community partition is possible with high probability as $n$ tends to infinity. This, on the other hand, is not possible in the sparse model, where the best that we can hope for is a partial recovery of the community labels. In particular, it is known that the Kesten-Stigum threshold is crucial for partial recovery. This threshold, which will be formally defined in the next section, can be interpreted as a signal-to-noise ratio, so partial recovery is possible only if this value is large enough. However, it does not provide information on the actual proportion of community labels that can be guessed correctly. Analysis of this optimal fraction, as well as an efficient algorithm for achieving this fraction are important questions in the theory underlying community detection. 

This problem is well-studied in \cite{beliefpropogation} by Mossel, Neeman, and Sly in the specific case that there are only two communities in the model. In this case, it was already known that partial recovery is possible if and only if the Kesten-Stigum threshold $\theta^2d > 1$. They were able to show that with a variant of the belief propagation algorithm, a provably optimal fraction of the vertices can be recovered efficiently given the stronger condition that $\theta^2d > C$ for some large constant $C$. Their proof consisted of three main steps. First, considering exact information versus noisy information on the $k$-th level of a regular tree, they showed that the probability of recovering the root label correctly tends to the same value as $k \rightarrow \infty$ in both cases. In particular, for this first step, they first bounded the probability of recovering the root label by calculating the probability by guessing according to the majority label on the $k$th level. Then, under certain situations, they showed that once the noisy and non-noisy probabilities are close, they in fact contract and are thus equal in the limit. This was achieved by analyzing the probabilities taking advantage of the recursive nature of the tree. For the second main step, the proof of this result can then be adapted to the case of  Galton-Watson trees with a Poisson number of children. The final step uses the fact that a small neighborhood of a vertex in the sparse stochastic block model can be coupled with such a Galton-Watson tree. This shows that a noisy belief propagation estimate can be amplified to give an accurate label prediction with optimal accuracy. Their result is the first to provide an algorithm that gives such an optimal recovery. The motivation of this paper is to generalize their results to any number of communities using similar methods. 

\subsection{Main Definitions}
In this section, we make the central definitions that will be used throughout the paper. Suppose we are working in the regime with $q \geq 3$ communities. Formally, we may define as our tree as a subset of $\mathbb{N}^*$, the set of finite strings of natural numbers, such that if $s \in T$, then any prefix of $s$ must also be in $T$. The condition that $T$ is $d$-regular can then be expressed as $\forall s \in T$, $|\{ n \in \mathbb{N}: sn \in T \}| = d$. With this formulation, the root is defined as the empty string. This is a formal definition borrowed from \cite{beliefpropogation}, but the result is the familiar $d$-regular rooted tree. First, we need to formally define the Kesten-Stigum threshold that is essential to partial recovery. Let $\rho$ be the root of the tree, and $\sigma$ be the true labels on the vertices of the tree defined as follows. We start with the label $\sigma_\rho$ uniformly distributed in $[q]$, and propagate down the tree according to some transition probabilities. With this simplifying assumptions described above in the introduction, we in particular have that the transition probability matrix $M$ has the form:
\[ M_{ij} \coloneqq \Prob{\sigma_v = j \middle\vert \sigma_u = i} = \begin{cases}
1-p & i=j \\
\frac{p}{q-1} & i\neq j
\end{cases} \]
where $v$ is a child of $u$. Now, we are interested in the second largest eigenvalue of this matrix, which can be computed to be 
\[ \lambda = 1 - \frac{pq}{q-1} \]
Now that we have a definition of the condition $\lambda^2d$, we can move to the definitions relevant to the probability of recovering the root correctly. Let $L_k(u)$ be the $k$-th level in the subtree rooted at $u$, and $\sigma_{L_k(u)}$ be the labels on this level. We let 
\[ X_u^{(m)}(i) \coloneqq \Prob{\sigma_u = i \middle \vert \sigma_{L_m(u)}} \]
be the posterior probability of the root vertex $u$ being labeled $i$ given the labels on the $m$th level of its subtree. Then, the maximum likelihood estimate of the root label given the $m$th sublevel can be defined as 
\[ \hat\sigma(m) \coloneqq \arg\max_i X_\rho^{(m)}(i) \]
which is the estimate that maximizes the probability of guessing the root label correctly, as suggested by the name. Moreover, this probability of obtaining the correct label can be explicitly expressed as 
\[ E_m \coloneqq \Prob{\hat\sigma(m) = \sigma_\rho} = \E{\max_i X_\rho^{(m)}(i)} \]
Now, we move to labels with some noise. We start with the noise matrix $\Delta = (\Delta_{ij}) \in [0,1]^{q\times q}$ where 
\[ \Delta_{ij} = \Prob{\tau_u = j \middle\vert \sigma_u = i} \]
We make two assumptions on $\Delta$ that will ultimately be related to properties of the selected black box algorithm discussed further in \cref{blackboxsection}. 
\begin{enumerate}
	\item $\sum_{i=1}^q \Delta_{ij} = 1$ for all $j \in [q]$. 
	\item $\Delta_{ii} \geq 1- \frac{1}{q}$ for all $i \in [q]$. 
\end{enumerate}
Relative to these noisy labels, we can similarly define the posterior probability vector $W_u$, maximum likelihood estimate $\hat\tau$ and probability of correct recovery $\tilde E_m$ as follows:
\begin{align*}
W_u^{(m)}(i) &\coloneqq \Prob{\sigma_u = i \middle\vert \tau_{L_m(u)}} \\
\hat\tau(m) &\coloneqq \arg\max W_\rho^{(m)}(i) \\
\tilde E_m &\coloneqq \Prob{\hat\tau(m) = \sigma_\rho} = \E{\max_i W_\rho^{(m)}}
\end{align*}
With these definitions in place, we can now precisely state the main result of the paper. 

\subsection{Main Results} 
\begin{lemmann}
	For any fixed $v \in G$, there is a coupling between $(G, \sigma')$ and $(T, \sigma')$ such that $(B(v, R), \sigma_{B(v, R)}') = (T_R, \sigma_R)$ a.a.s.
\end{lemmann}

\begin{lemmann}
	For any fixed $v \in G$, there is a coupling between $(G, \tau')$ and $(T, \tau)$ such that $(B(v, R), \tau_{B(v, R)}') = (T_R, \tau_R)$ a.a.s. where $\tau'$ are the labels produced by the black box algorithm. 
\end{lemmann}

The above two lemmas provide the relation between the stochastic block model generated graph and the tree. They imply that the difficulty of recovering a label on the graph is sandwiched between the noisy and non-noisy tree recovery problems, which under certain conditions are equally as difficult. These conditions are presented in the main theorem below. 

\begin{thmnn}\label{mainthm}
	There exists a $C^*(q)$ such that if $\lambda^2d > C^*(q)$, then 
	\[ \lim_{m \rightarrow \infty} E_m = \lim_{m \rightarrow \infty} \tilde E_m \]	
\end{thmnn}

The proof of this theorem will follow the same structure as the result we are generalizing in \cite{beliefpropogation}. First, assuming that $\lambda$ is bounded away from 1 and $\lambda^2d > C^*(q) \log d$, we show that both $E_m$ and $\tilde E_m$ are quite close to 1. In particular, this implies that they must be quite close to each other as well. This is achieved by calculating the probability of correct recovery using a easier method -- the simple majority. Then, we use the nature of the tree to recursively analyze $X_\rho$ and $W_\rho$ to show a contraction of the difference between the probability vectors in expectation as the depth of the tree increases. This will allow us to conclude that in the limit, these entry-wise maximums are equal. This proof is then extended through various adjustments to capture the cases where $\lambda$ is close to 1, and where $C^*(q) < \lambda^2d < C(q)\log d$, which proves the entire theorem. At each step, we also extend from a fixed tree to a random tree by conditioning on the children of the root. 

Note that by symmetry of the communities, the probability of recovering the root correctly is the same as the probability of recovering the root correctly given the label of the root. More formally, we have that
\[ \Prob{\hat\sigma(m) = \sigma_\rho} = \Prob{\hat\sigma(m) = i \middle\vert \sigma_\rho = i} \]
Because of this, we may assume without loss of generality that $\sigma_\rho = 1$ throughout the entire paper. In particular, we execute each of the above steps under this additional assumption, which will suffice to give the desired result. This is an assumption to keep in mind whenever calculations are made regarding probabilities on the tree.

These three results together prove that the provided algorithm recovers an optimal fraction of community labels under the given conditions, which is the full generalization of \cite{beliefpropogation}. 

\section{Simple Majority Method}\label{Simple Majority Method}
In this section, we will provide a lower bound on the optimal probability $E_k$ by computing the probability of successful reconstruction with a simpler method of guessing -- namely the simple majority at level $k$. This probability is easier to compute, as we can estimate the number of leaves with each label at the given level, and hence estimate the probability of the majority being labeled 1. 

The outline of this section is as follows. We first define the counts for each label on the $k$th level of the tree, for both noisy and non-noisy labels. Then, using the recursive nature of the tree, we can compute the expectation and variance of these random variables. By applying Chebyshev's Inequality we obtain a concentration of each random variable about their respective means. Since the expected number of vertices labeled 1 is larger than the rest of the labels, if this concentration is tight enough, we can guarantee that the count of label 1 will be larger than the counts of all the other labels with high probability. This will show that we guess the root correctly using the simple majority with at least this large probability. 

\subsection{Non-noisy Setting}
To begin, we find the expectation and variance of each label count on a fixed level $k$.
\begin{defn}\label{majoritydef}
	We define variables for the count of each label on the $k$th level given that the root has label $1$. 
	\begin{align*}
	Z_{u, k} &= \sum_{v \in L_k(u)} \mathbbm{1}(\sigma_u = 1 | \sigma_\rho = 1) \\
	Y_{u, k}^{(i)} &= \sum_{v \in L_k(u)} \mathbbm{1}(\sigma_u = i | \sigma_\rho = 1) \\
	S_v &= 2 \cdot \mathbbm{1}(\sigma_v = 1 | \sigma_{\rho} = 1) - 1 \\
	S_{u, k} &= \sum_{v \in L_k(u)} S_v
	\end{align*}
\end{defn}
Notice that $S_{\rho, k} = Z_{\rho, k} - \sum_{i \neq 1} Y_{\rho, k}^{(i)}$ so that by symmetry, we have $\mathbb{E}[Z_{\rho, k}] = \dfrac{d^k + \mathbb{E}[S_{\rho, k}]}{2}$ and $\mathbb{E}[Y_{\rho, k}^{(i)}] = \dfrac{d^k - \mathbb{E}[S_{\rho, k}]}{2(q-1)}$ for every $i$.  In order to compute these two expectations, we compute the expectation of $S_v$ for $v\in L_k(\rho)$ and use linearity of expectation. By the symmetrical nature of the tree being 1, conditioned on the label of the root we  have that $S_v$ are i.i.d. for $v \in L_k(\rho)$. 

\begin{lemma}\label{expectationS}
	For any $v \in L_k(\rho)$, \[ \mathbb{E}[S_v]= \left( 2 - \frac{2}{q} \right)\lambda^k + \frac{2}{q}-1 \]
\end{lemma}

\begin{proof}	
Suppose that $\sigma_\rho = 1$ and fix $v \in L_k(\rho)$. We can write $\mathbb{E}[S_v] = \mathbb{P}(\sigma_v = 1|\sigma_\rho = 1) - \mathbb{P}(\sigma_v \neq 1|\sigma_\rho=1)$. Let $w \in L_{k-1}(\rho)$ be the parent of $v$ in the tree. Conditioning on the state of $w$, we can express the above as
\begin{align*}
\mathbb{E}[S_v] &= \mathbb{P}(\sigma_v =1| \sigma_\rho=1, \sigma_w = 1)\mathbb{P}(\sigma_w = 1|\sigma_\rho = 1) + \mathbb{P}(\sigma_v =1| \sigma_\rho=1, \sigma_w \neq 1)\mathbb{P}(\sigma_w \neq 1|\sigma_\rho = 1)  \\
& \qquad -\mathbb{P}(\sigma_v \neq1| \sigma_\rho=1, \sigma_w = 1)\mathbb{P}(\sigma_w = 1|\sigma_\rho = 1)  - \mathbb{P}(\sigma_v \neq1| \sigma_\rho=1, \sigma_w \neq 1)\mathbb{P}(\sigma_w \neq 1|\sigma_\rho = 1) \\
&= \mathbb{P}(\sigma_v =1| \sigma_w = 1)\mathbb{P}(\sigma_w = 1|\sigma_\rho = 1) + \mathbb{P}(\sigma_v =1|  \sigma_w \neq 1)\mathbb{P}(\sigma_w \neq 1|\sigma_\rho = 1)  \\
& \qquad -\mathbb{P}(\sigma_v \neq1|\sigma_w = 1)\mathbb{P}(\sigma_w = 1|\sigma_\rho = 1)  - \mathbb{P}(\sigma_v \neq1|  \sigma_w \neq 1)\mathbb{P}(\sigma_w \neq 1|\sigma_\rho = 1) \\
&= (1-p)\cdot \frac{1 + \mathbb{E}[S_w]}{2} + \frac{p}{q-1}\cdot \frac{1-\mathbb{E}[S_w]}{2} - p\cdot \frac{1 + \mathbb{E}[S_w]}{2} - \left(1-\frac{p}{q-1}\right) \cdot \frac{1-\mathbb{E}[S_w]}{2} \\
&= (1-2p)\cdot \frac{1+\mathbb{E}[S_w]}{2} +\left (\frac{2p}{q-1} - 1\right) \cdot \frac{1-\mathbb{E}[S_w]}{2}
\end{align*}
where we have used the fact that $\mathbb{P}(\sigma_w = 1 | \sigma_\rho = 1) = \dfrac{1+\mathbb{E}[S_w]}{2}$ and $\mathbb{P}(\sigma_w \neq 1 | \sigma_\rho = 1) = \dfrac{1-\mathbb{E}[S_w]}{2}$. We simplify the coefficients into a form that will be easier to analyze. In particular, we express them as a multiple of $\lambda$. 
\begin{align*}
\frac{1-2p}{\lambda} &= \frac{(1-2p)(q-1)}{q-1-pq} = \frac{q-1-2pq+2p}{q-1-pq} = 1 - \frac{pq-2p}{q-1-pq}
\end{align*}
\[ \frac{\frac{2p}{q-1}-1}{\lambda} = \frac{(2p-q+1)(q-1)}{(q-1-pq)(q-1)} = \frac{2p-q+1}{q-1-pq} = -1 - \frac{pq-2p}{q-1-pq} \]
Let $a = \frac{pq-2p}{q-1-pq}$ and $A_k = \mathbb{E}[S_v]$ for any $v \in L_k(\rho)$. Our above expression then becomes:
\begin{align*}
A_k &= (1-a)\lambda\cdot \frac{1+A_{k-1}}{2} - (1+a)\lambda\cdot \frac{1-A_{k-1}}{2} \\
&= \frac{\lambda}{2} \left( 1 + A_{k-1} -a-aA_{k-1} - 1 + A_{k-1}-a+aA_{k-1}\right) \\
&= \frac{\lambda}{2}(2A_{k-1}-2a) \\
&= \lambda(A_{k-1}-a)
\end{align*}
This is a linear recurrence relation, so along with the initial condition that $A_0 = 1$, we can solve for the general expression. Let $B_k = A_k - \frac{\lambda}{\lambda-1}a$ with initial condition $B_0 = 1 - \frac{\lambda}{\lambda-1}a$. We can then transform the above into
\begin{align*}
A_k &= \lambda(A_{k-1}-a) \\
A_k - \frac{\lambda}{\lambda-1}a &= \lambda A_{k-1}  - \frac{\lambda^2-\lambda}{\lambda - 1}a - \frac{\lambda}{\lambda-1}a \\
B_k &= \lambda(A_{k-1} - \frac{\lambda}{\lambda-1}a) \\
B_k &= \lambda B_{k-1}
\end{align*}
This gives the geometric general form $B_k = \left( 1 - \frac{\lambda}{\lambda-1}a \right) \lambda^k$. Converting back to the desired $A_k$, we find that $A_k = \left( 1 - \frac{\lambda}{\lambda-1}a \right) \lambda^k + \frac{\lambda}{\lambda-1}a$. We can simplify this expression by recalling the definition of $a$.
\begin{align*}
\frac{\lambda}{\lambda-1}a &= \frac{q-1-pq}{q-1} \cdot \frac{q-1}{-pq} \cdot \frac{pq-2p}{q-1-pq} = -\frac{pq-2p}{pq} = \frac{2}{q}-1
\end{align*}
Substituting this into the above expression for $A_k$ gives 
\[ A_k = \left( 2 - \frac{2}{q} \right)\lambda^k + \frac{2}{q}-1 \]
\end{proof}

We can now return to our original goal of calculating the expectations of $Z_{\rho, k}$ and $Y_{\rho, k}^{(i)}$.

\begin{lemma}\label{expectationYZ}
	For any $k$ and any $i$, we have that
	\begin{align*}
	\mathbb{E}[Z_{\rho, k}] &= \left(1-\frac{1}{q}\right)\lambda^kd^k + \frac{d^k}{q} \\
	\mathbb{E}[Y_{\rho, k}^{(i)}] &= -\frac{1}{q} \cdot \lambda^k d^k+ \frac{d^k}{q} 
	\end{align*}
\end{lemma}

\begin{proof}
By the remarks following Definition 1.1, we have that
\begin{align*}
\mathbb{E}[Z_{\rho, k}] & = \dfrac{d^k + \mathbb{E}[S_{\rho, k}]}{2} = \dfrac{d^k + \sum_{v \in L_k(\rho)} \mathbb{E}[S_v]}{2} = \frac{d^k + d^k\left[\left(2-\frac{2}{q}\right)\lambda^k + \frac{2}{q}-1\right] }{2} \\
&= \left(1-\frac{1}{q}\right)\lambda^kd^k + \frac{d^k}{q} \\
\mathbb{E}[Y_{\rho, k}^{(i)}] & = \dfrac{d^k - \mathbb{E}[S_{\rho, k}]}{2(q-1)} = \dfrac{d^k - \sum_{v \in L_k(\rho)} \mathbb{E}[S_v]}{2(q-1)} = \frac{d^k - d^k\left[\left(2-\frac{2}{q}\right)\lambda^k + \frac{2}{q}-1\right] }{2(q-1)} \\
&= -\frac{1}{q} \cdot \lambda^k d^k+ \frac{d^k}{q} 
\end{align*}
\end{proof}

We now would like to compute the variance of these two random variables. To make the computations more friendly, we define the following normalization.

\begin{defn}\label{normalizedDef}
	\begin{gather*}
	\tilde Z_{\rho, k} = Z_{\rho, k} - \frac{d^k}{q} \\
	\tilde Y_{\rho, k}^{(i)} = Y_{\rho, k}^{(i)} - \frac{d^k}{q} 
	\end{gather*}
\end{defn}
As a direct consequence of Lemma 1.2, we have that
\begin{align*}
\mathbb{E}[\tilde Z_{\rho, k} &= \left( 1 - \frac{1}{q} \right) \lambda^kd^k \\
\mathbb{E}[\tilde Y_{\rho, k}^{(i)}] &= -\frac{1}{q} \cdot \lambda^k d^k
\end{align*}
Moreover, since these differ from our original variables by a constant, they have the same variance, which we bound in the next lemma.

\begin{lemma}\label{varianceYZ}
	For any $k$ and any $i$, we have that
	\[\Var{Z_{\rho, k}}, \Var{Y_{\rho, k}^{(i)}} < Rd^k \frac{(\lambda^2 d)^k - 1}{\lambda^2 d-1}\]
\end{lemma}

\begin{proof}
 Let $C_k = \text{Var}(\tilde Z_{\rho, k})$ and $D_k = \text{Var}(\tilde Y_{\rho, k}^{(i)})$. We use the recursive nature of the tree to produce recurrence relations between these two sequences. To do this, we condition on the labels at the first level of the tree $\sigma_u$ for $u \in L_1(\rho)$. Writing out the decomposition of the variance, we get
\begin{align*}
C_k &= \Var{\tilde Z_{\rho, k}} = \E{\Var{\tilde Z_{\rho, k} \middle\vert \sigma_{u_1}, \ldots, \sigma_{u_d}}} + \Var{\E{\tilde Z_{\rho, k} \middle \vert \sigma_{u_1}, \ldots, \sigma_{u_d}}}
\end{align*}
We consider each term separately. First, the conditional variance of $\tilde Z_{\rho, k}$. 
\begin{align*}
\Var{\tilde Z_{\rho, k} \middle\vert \sigma_{u_1}, \ldots, \sigma_{u_d}} &= \Var{\sum_{i=1}^d \tilde Z_{u_i, k-1} \middle \vert \sigma_{u_1}, \ldots, \sigma_{u_d}} = \sum_{i=1}^d \Var{\tilde Z_{u_i, k-1} \middle \vert \sigma_{u_i}}
\end{align*}
where the second equality follows from the mutual independence of disjoint subtrees. If $\sigma_{u_i} = 1$ then the conditional variable $\tilde Z_{u_i, k-1} \vert \sigma_{u_i} \sim \tilde Z_{\rho, k-1}$ and thus has variance $C_{k-1}$. On the other hand, if $\sigma_{u_i} \neq 1$, then the conditional variable $\tilde Z_{u_i, k-1} \vert \sigma_{u_i} \sim \tilde Y_{\rho, k-1}^{(2)}$ and thus has variance $D_{k-1}$. Putting this together, we find that
\[ \E{\Var{\tilde Z_{\rho, k} \middle\vert \sigma_{u_1}, \ldots, \sigma_{u_d}}} = \sum_{i=1}^d \E{\Var{\tilde Z_{u_i, k-1} \middle\vert \sigma_{u_i}}} = d(1-p)C_{k-1} + dpD_{k-1} \]
Next, the conditional expectation of $\tilde Z_{\rho, k}$ can be computed similarly, as 
\begin{align*}
\E{\tilde Z_{\rho, k} \middle\vert \sigma_{u_1}, \ldots, \sigma_{u_d}} &= \E{\sum_{i=1}^d \tilde Z_{u_i, k-1} \middle \vert \sigma_{u_1}, \ldots, \sigma_{u_d}} = \sum_{i=1}^d \E{\tilde Z_{u_i, k-1} \middle \vert \sigma_{u_i}}
\end{align*}
As before, if $\sigma_{u_i} = 1$ then the variable is distributed like $\tilde Z_{\rho, k-1}$ and thus has expectation $\left(1-1/q \right)(\lambda d)^{k-1}$. If $\sigma_{u_i} \neq 1$, then the variable is distributed like $\tilde Y_{\rho, k-1}^{(2)}$ and has expectation $-(\lambda d)^{k-1}/q$. Thus, we can see that the conditional distribution of the random variable follows
\begin{equation*}
\E{\tilde Z_{u_i, k-1} | \sigma_{u_i}} \sim \left(\text{Ber}(1-p) - \frac{1}{q} \right)\cdot (\lambda d)^{k-1}
\end{equation*}
Putting this together, using the conditional independence of the subtrees we see that 
\[  \Var{\E{\tilde Z_{\rho, k} \middle \vert \sigma_{u_1}, \ldots, \sigma_{u_d}}} =  \sum_{i=1}^d \Var{\left(\text{Ber}(1-p) - \frac{1}{q} \right)\cdot (\lambda d)^{k-1}} = d(\lambda d)^{2k-2}p(1-p) \]
The first recurrence relation all together is of the form
\[ C_k = d(1-p)C_{k-1} + dpD_{k-1} + d(\lambda d)^{2k-2}p(1-p) \]

We similarly compute the recurrence relation for $D_k$. Decomposing the variance, we have that 
\begin{align*}
D_k &= \Var{\tilde Y_{\rho, k}^{(i)}} = \E{\Var{\tilde Y_{\rho, k}^{(i)} \middle\vert \sigma_{u_1}, \ldots, \sigma_{u_d}}} + \Var{\E{\tilde Y_{\rho, k}^{(i)} \middle \vert \sigma_{u_1}, \ldots, \sigma_{u_d}}}
\end{align*}
Simplifying the conditional variance we have that
\begin{align*}
\Var{\tilde Y_{\rho, k}^{(i)} \middle\vert \sigma_{u_1}, \ldots, \sigma_{u_d}} &= \Var{\sum_{j=1}^d \tilde Y_{u_j, k-1}^{(i)} \middle \vert \sigma_{u_1}, \ldots, \sigma_{u_d}} = \sum_{j=1}^d \Var{\tilde Y_{u_j, k-1}^{(i)} \middle \vert \sigma_{u_j}}
\end{align*}
In this case, the conditional distribution follows that of $\tilde Z_{\rho, k-1}$ when $\sigma_{u_j} = i$, which happens with probability $\frac{p}{q-1}$. Otherwise, the distribution follows that of $\tilde Y_{\rho, k-1}^{(i)}$. This gives the first term as
\[ \E{\Var{\tilde Y_{\rho, k}^{(i)} \middle\vert \sigma_{u_1}, \ldots, \sigma_{u_d}}} = \sum_{j=1}^d \E{\Var{\tilde Y_{u_j, k-1}^{(i)} \middle\vert \sigma_{u_j}}} = d \cdot \frac{p}{q-1} \cdot C_{k-1} + d\cdot \left( 1-\frac{p}{q-1} \right)\cdot D_{k-1} \]
Simplifying the conditional expectation, we have that
\begin{align*}
\E{\tilde Y_{\rho, k}^{(i)} \middle\vert \sigma_{u_1}, \ldots, \sigma_{u_d}} &= \E{\sum_{j=1}^d \tilde Y_{u_i, k-1}^{(i)} \middle \vert \sigma_{u_1}, \ldots, \sigma_{u_d}} = \sum_{j=1}^d \E{\tilde Y_{u_i, k-1}^{(i)} \middle \vert \sigma_{u_j}}
\end{align*}
Once again, if $\sigma_{u_j} = i$ then this variable has distribution of $\tilde Z_{\rho, k-1}$ and has expectation $(1-1/q)(\lambda d)^{k-1}$. If $\sigma_{u_i} \neq 1$, then the variable is distributed like $\tilde Y_{\rho, k-1}^{(2)}$ and has expectation $-(\lambda d)^{k-1}/q$. Thus, we can see that the conditional distribution of the random variable follows
\[ \tilde Y_{u_j, k-1}^{(i)} | \sigma_{u_j} \sim \left(\text{Ber}\left(\frac{p}{q-1}\right) - \frac{1}{q} \right)\cdot (\lambda d)^{k-1} \]
Putting this together, we see that 
\[  \Var{\E{\tilde Y_{\rho, k}^{(i)} \middle \vert \sigma_{u_1}, \ldots, \sigma_{u_d}}} =  \sum_{j=1}^d \Var{\left(\text{Ber}\left(\frac{p}{q-1}\right) - \frac{1}{q} \right)\cdot (\lambda d)^{k-1}} = d(\lambda d)^{2k-2}\left(\frac{p}{q-1}\right)\left(1-\frac{p}{q-1}\right) \]
The first recurrence relation all together is of the form
\[ D_k = d\cdot \frac{p}{q-1} \cdot C_{k-1} + d\left(1-\frac{p}{q-1}\right) D_{k-1} + d(\lambda d)^{2k-2}\left(\frac{p}{q-1}\right)\left(1-\frac{p}{q-1}\right)\]

The precise analysis of the recurrence is difficult, but we only need an upper bound on the variances of the variables $Z_{\rho, k}$ and $Y_{\rho, k}^{(i)}$, so it suffices to compute the precise value of a linear combination of $C_k$ and $D_k$. In particular, since $C_k$ and $D_k$ are variances and thus positive, we have that $C_k, D_k < mC_k + nD_k$ for $m, n \geq 1$. We consider $F_k = C_k + (q-1)D_k$ by adding $q-1$ times the second recurrence relation to the first. This gives the expression
\begin{align*}
F_k &= C_k + (q-1)D_k \\
&= d(1-p)C_{k-1} + dpD_{k-1} + d(\lambda d)^{2k-2}p(1-p)  \\
&\qquad +d\cdot p\cdot C_{k-1} + d\left(q-1-p\right) D_{k-1} + d(\lambda d)^{2k-2}p\left(1-\frac{p}{q-1}\right) \\
&= dC_{k-1} + d(q-1)D_{k-1}+d(\lambda d)^{2k-2}p\left(2-p-\frac{p}{q-1}\right) \\
&= dF_{k-1} + d(\lambda d)^{2k-2}p\left(2-p-\frac{p}{q-1}\right) 
\end{align*}
For the following calculation, let $R = p\left(2-p-\frac{p}{q-1}\right)$. Since $C_0 = D_0 = 0 \implies F_0 = 0$, the general solution to our recurrence relation is simply
\begin{align*}
F_k = \sum_{l=1}^k d(\lambda d)^{2l-2}Rd^{k-l} = Rd^k \sum_{l=1}^k (\lambda^2 d)^{l-1} = Rd^k \frac{(\lambda^2 d)^k - 1}{\lambda^2 d-1}
\end{align*}
In particular, we can conclude by the above remark that 
\[ \Var{Z_{\rho, k}}, \Var{Y_{\rho, k}^{(i)}} < Rd^k \frac{(\lambda^2 d)^k - 1}{\lambda^2 d-1} \]
\end{proof}

In summary, we have so far computed that the random variables $Z_{\rho, k}$ and $Y_{\rho, k}^{(i)}$ have distributions satisfying the following:
\begin{align*}
&\E{Z_{\rho, k}} = \left(1-\frac{1}{q}\right) \cdot \lambda^k d^k+ \frac{d^k}{q}, &&\Var{Z_{\rho, k}} < Rd^k \frac{(\lambda^2 d)^k - 1}{\lambda^2 d-1}  \\
& \E{Y_{\rho, k}^{(i)}} = -\frac{1}{q}\cdot \lambda^k d^k+ \frac{d^k}{q}, &&\Var{Y_{\rho, k}^{(i)}} < Rd^k \frac{(\lambda^2 d)^k - 1}{\lambda^2 d-1}  
\end{align*}
 
\subsection{Noisy Setting}

Next, we make the analogous computations in the case of noisy labels, and find that the corresponding variables have similar expectation and variance. 

\begin{defn}\label{noisyDef}
	We make analogous definitions in the noisy regime for the count of each label on the $k$th level.
	\begin{gather*}
	Z_{u, k}' = \sum_{v \in L_k(u)} \mathbbm{1}(\tau_u = 1 | \sigma_\rho = 1) \\
	Y_{u, k}^{(i)'} = \sum_{v \in L_k(u)} \mathbbm{1}(\tau_u = i | \sigma_\rho = 1) 
	\end{gather*}
\end{defn}
We would like to derive similar results for the variables $Z_{\rho, k}'$ and $Y_{\rho, k}^{(i)'}$ for the label counts based on noisy labels with probability of error $\delta$. We furthermore normalize again so that $\tilde Z_{\rho, k}' = Z_{\rho, k}' - \frac{d^k}{q}$ and $\tilde Y_{\rho, k}^{(i)'} = Y_{\rho, k}^{(i)'} - \frac{d^k}{q}$.  We analyze these variables by conditioning on their non-noisy counterparts. 

\begin{lemma}\label{noisyExpVar}
	For any $k$ and any $i$, we have the following equalities and estimates:
	\begin{align*}
	\E{Z_{\rho, k}'} &= \left( \Delta_{11} - \frac{1}{q} \right)(\lambda d)^k + \frac{d^k}{q} \\
	\Var{Z_{\rho, k}'} &< O(d^k) + \left(\sum_{j=1}^q \Delta_{j1}^2\right) Rd^k \frac{(\lambda^2 d)^k-1}{\lambda^2 d-1} \\
	\E{Y_{\rho, k}^{(i)'}} &=  \left(\Delta_{1i} - \frac{1}{q}\right)\left(-\frac{1}{q}\right)(\lambda d)^k + \frac{d^k}{q} \\
	\Var{Y_{\rho, k}^{(i)'}} &< O(d^k) + \left(\sum_{j=1}^q \Delta_{ji}^2\right) Rd^k \frac{(\lambda^2 d)^k-1}{\lambda^2 d-1}
	\end{align*}
\end{lemma}

\begin{proof}
Conditioned on $Z_{\rho, k}$ and $Y_{\rho, k}^{(i)}$, in the $k$th level of neighbors from $\rho$ there are $ Z_{\rho, k}$ nodes with label 1 and $Y_{\rho, k}^{(i)}$ with label $i$ for each $i$. For each of the labels that are currently 1, there is a $\Delta_{11}$ chance that the noisy label will remain 1, which is equal in distribution to $\text{Ber}(\Delta_{11})$. For the labels that are currently not 1, there is a $\Delta_{i1}$ that the noisy label will change to precisely the label 1, which has the same distribution as $\text{Ber}(\Delta_{i1})$. Thus, we have that the conditional expectation and variance have the expressions
\begin{align*}
\E{Z_{\rho, k}' \middle\vert  Z_{\rho, k}, Y_{\rho, k}^{(i)}} &= \sum_{\sigma_v = 1} \E{\text{Ber}(\Delta_{11})} + \sum_{i=2}^q \sum_{\sigma_v = i} \E{\text{Ber}\left(\Delta_{i1}\right)} \\
&= \Delta_{11} \cdot Z_{\rho, k} + \sum_{i=2}^q \Delta_{i1} \cdot Y_{\rho, k}^{(i)} \\
\Var{Z_{\rho, k} \middle \vert  Z_{\rho, k}, Y_{\rho, k}^{(i)}} &= \sum_{\sigma_v = 1} \Var{\text{Ber}(\Delta_{11})} + \sum_{i=2}^q \sum_{\sigma_v =i} \Var{\text{Ber}\left(\Delta_{i1}\right)} \\
&= \Delta_{11}(1-\Delta_{11})\left( Z_{\rho, k}\right) + \sum_{i=2}^q \Delta_{i1}  \left(1-\Delta_{i1}\right)  \left(Y_{\rho, k}^{(i)}\right) 
\end{align*}
It follows that the unconditional distribution of $Z_{\rho, k}'$ satisfy
\begin{align*}
\E{Z_{\rho, k}'} &= \E{\E{Z_{\rho, k}' \middle\vert  Z_{\rho, k}, Y_{\rho, k}^{(i)}}} = \E{\Delta_{11} \cdot Z_{\rho, k} + \sum_{i=2}^q \Delta_{i1} \cdot Y_{\rho, k}^{(i)} } \\
&= \Delta_{11} \cdot \left( \left(1-\frac{1}{q} \right)\lambda^k d^k + \frac{d^k}{q} \right) + \sum_{i=2}^q \Delta_{i1} \cdot \left( - \frac{1}{q} \lambda^k d^k + \frac{d^k}{q} \right) \\
\intertext{Grouping together the terms that contain $\lambda^k d^k$ and $\frac{d^k}{q}$ we get}
&= \left[ \Delta_{11} \left( 1 - \frac{1}{q} \right) - \frac{1}{q} \sum_{i=2}^q \Delta_{i1} \right] \lambda^k d^k + \sum_{i=1}^q \Delta_{i1} \cdot \frac{d^k}{q} \\
\intertext{By assumption $\sum_{i=1}^q \Delta_{i1} = 1$ we get the final simplified form}
&= \left( \Delta_{11}- \frac{1}{q} \right) \lambda^k d^k + \frac{d^k}{q}
\end{align*}
For the variance, we start by expanding by conditioning on the same variables. 
\begin{align*}
\Var{Z_{\rho, k}'} &= \E{\Var{Z_{\rho, k}' \middle\vert \tilde Z_{\rho, k}}} + \Var{\E{Z_{\rho, k}' \middle\vert \tilde Z_{\rho, k}}} \\
&< O(d^k) + \Var{\Delta_{11} \cdot Z_{\rho, k} + \sum_{i=2}^q \Delta_{i1} \cdot Y_{\rho, k}^{(i)} }
\intertext{Since the $Z_{\rho, k}$ and $Y_{\rho, k}^{(i)}$ are correlated negatively, and all coefficients are positive, the variance of the sum is bounded above by the sum of the variances. Thus, we get}
& < O(d^k) + \Delta_{11}^2 \Var{Z_{\rho, k}} + \sum_{i=2}^q \Delta_{i1}^2 \Var{Y_{\rho, k}^{(i)}} \\
\intertext{Recall that we bounded the variances of $Z_{rho, k}$ and $Y_{\rho, k}^{(i)}$ by the same upper bound, so grouping together like terms we get the resulting upper bound}
& < O(d^k) + \left(\sum_{i=1}^q \Delta_{i1}^2 \right) R d^k \frac{(\lambda^2 d)^k - 1}{\lambda^2d - 1} \end{align*}
By the exact argument as above with the $Y_{\rho, k}^{(i)'}$ instead of $Z_{\rho, k}'$, we get the analogous statements about the distribution of $Y_{\rho, k}^{(i)'}$.
\begin{align*}
\E{Y_{\rho, k}^{(i)'}} &=  \left(\Delta_{1i}-\frac{1}{q}\right)(\lambda d)^k + \frac{d^k}{q} \\
\Var{Y_{\rho, k}^{(i)'}} &< O(d^k) + \left(\sum_{j=1}^q \Delta_{ji}^2\right)^2 Rd^k \frac{(\lambda^2 d)^k-1}{\lambda^2 d-1}
\end{align*}
\end{proof}

\subsection{Majority Calculation}

Equipped with the above lemmas, we are now ready to show that when $\lambda^2 d$ is large, the simple majority estimator does quite well as the depth of the tree increases. In particular, we are interested in the event \[M_k = \{ \text{1 is the most frequent label on level } k \} = \{ Z_{\rho, k} > Y_{\rho, k}^{(i)} \ \forall i\neq 1 \}\] conditioned on the event that the root is 1. Recall that all of our calculations have been performed under the assumption that $\sigma_\rho = 1$, so the conditioning will be omitted below as well. 

\begin{prop}\label{majority}
	For some constant $C = C(q)$, 
	\[ \liminf_{k \rightarrow \infty}\Prob{M_k} > 1 - \frac{C}{\lambda^2d-1} \]
\end{prop}
\begin{proof}
For this setting, denote the events 
\[ M_k^{(1)} = \left\lbrace Z_{\rho, k} > \left(\frac{1}{2} - \frac{1}{q}\right)(\lambda d)^k + \frac{d^k}{q} \right\rbrace,  M_k^{(i)} = \left\lbrace Y_{\rho, k}^{(i)} > \left(\frac{1}{2} - \frac{1}{q}\right)(\lambda d)^k + \frac{d^k}{q} \right\rbrace \]
Then event $M_k^{(1)} \cap \bigcap_{i \neq 1} \overline{M_k^{(i)}}$ is a subset of $M_k$, as we have bounded $Z_{\rho, k}$ below and all of the $Y_{\rho, k}^{(i)}$ above by the same value. Thus, we have that
\begin{align*}
\Prob{M_k} &\geq \Prob{M_k^{(1)} \cap \bigcap_{i \neq 1} \overline{M_k^{(i)}} } = \Prob{M_k^{(1)} \cap \left(\bigcup_{i \neq 1} M_k^{(i)}\right)^c } \\
&\geq \Prob{M_k^{(1)} } - \Prob{\bigcup_{i \neq 1} M_k^{(i)}} \\
&\geq \Prob{M_k^{(1)} } - \sum_{i \neq 1} \Prob{M_k^{(i)} } 
\end{align*}
It only remains to bound $M_k^{(i)}$ for every $i$. Since we have computed the expectations and variances of the relevant random variables above, we can apply Chebyshev's Inequality to obtain the desired bound. First, for $i=1$, we have that
\begin{align*}
\Prob{M_k^{(1)} } &= \Prob{Z_{\rho, k} > \left(\frac{1}{2} - \frac{1}{q}\right)(\lambda d)^k + \frac{d^k}{q} } \geq \Prob{\abs{Z_{\rho, k} - \E{Z_{\rho, k}}} < \frac{(\lambda d)^k}{2}} \\
&= 1- \Prob{\abs{Z_{\rho, k} - \E{Z_{\rho, k}}} \geq \frac{(\lambda d)^k}{2}} \geq 1 - \frac{4\Var{Z_{\rho, k}}}{(\lambda d)^{2k}} \\
&> 1 - \frac{4Rd^k((\lambda^2d)^k-1)}{(\lambda d)^{2k}(\lambda^2d-1)} \xrightarrow{k \rightarrow \infty} 1 - \frac{4R}{\lambda^2d-1}
\end{align*}
For $i \neq 1$, we have the similar computation that
\begin{align*}
\Prob{M_k^{(i)} } &= \Prob{Y_{\rho, k}^{(i)} > \left(\frac{1}{2} - \frac{1}{q}\right)(\lambda d)^k + \frac{d^k}{q} } \leq \Prob{\abs{Y_{\rho, k}^{(i)} - \E{Y_{\rho, k}}^{(i)}} > \frac{(\lambda d)^k}{2}} \\
& \leq \frac{4\Var{Y_{\rho, k}^{(i)}}}{(\lambda d)^{2k}} <\frac{4Rd^k((\lambda^2d)^k-1)}{(\lambda d)^{2k}(\lambda^2d-1)}  \xrightarrow{k \rightarrow \infty} \frac{4R}{\lambda^2d-1}
\end{align*}
Putting it all together, we find that 
\[ \liminf_{k \rightarrow \infty}\Prob{M_k} > 1 - \frac{4R}{\lambda^2d-1} - \sum_{i \neq 1} \frac{4R}{\lambda^2d-1} = 1 - \frac{4Rq}{\lambda^2d-1} \]
We conclude the analysis of the non-noisy case by checking that the numerator is bounded even as $\lambda^2d$ increases. Recall that $R = p\left(2-p-\frac{p}{q-1}\right)$, and $\lambda = 1 - \frac{pq}{q-1} \implies p = \frac{(1-\lambda)(q-1)}{q}$
\begin{align*}
4Rq &= 4pq\left(2-p-\frac{p}{q-1}\right) = 4(1-\lambda)(q-1)\left(2-p-\frac{p}{q-1}\right) \\
&= 4(1-\lambda)(q-1)\left(\frac{2q-2-pq+p - p}{q-1}\right) \\
&= 4(1-\lambda)(q-1)\left(\frac{2q-2-pq}{q-1}\right) \\
&= 4(1-\lambda)(q-1)(1+\lambda) \\
&= 4(q-1)(1-\lambda^2) \\
&\leq 4(q-1)
\end{align*}
Thus, we indeed have the desired behavior. 
\end{proof}

We next repeat the analysis for the noisy labels, and show that the same bound holds for the event $N_k = \{ Z_{\rho, k}' > Y_{\rho, k}^{(i)'} \ \forall i \}$.

\begin{prop}\label{noisyMajority}
	There exists a constant $C = C(q)$ such that
	\[ \liminf_{k \rightarrow \infty}\Prob{N_k} >  1 - \frac{C(q)}{\lambda^2d-1} \]
\end{prop}
 \begin{proof}
 Recall that $\E{Z_{\rho, k}'} = \left(\Delta_{11} - \frac{1}{q} \right) )\lambda^kd^k + \frac{d^k}{q}$. Using the assumption that $\Delta_ii \geq 1 - \frac{1}{q}$, we know that this expectation is at least $\left(1 - \frac{2}{q}\right)(\lambda d)^k + \frac{d^k}{q}$. Similarly, we know the off diagonal terms of $\Delta$ are at most $\frac{1}{q}$, so we have that $\E{Y_{\rho, k}^{(i)'}}$ is at most $\frac{d^k}{q}$. Define the events 
\begin{gather*}
N_k^{(1)} = \left\lbrace Z_{\rho, k}' > \left(\frac{1}{2} - \frac{1}{q}\right)(\lambda d)^k + \frac{d^k}{q} \right\rbrace\\
N_k^{(i)} = \left\lbrace Y_{\rho, k}^{(i)'} > \left(\frac{1}{2} - \frac{1}{q}\right)(\lambda d)^k + \frac{d^k}{q} \right\rbrace 
\end{gather*}
By the same computation as above we know that 
\[ \Prob{N_k} \geq \Prob{N_k^{(1)}} - \sum_{i \neq 1} \Prob{N_k^{(i)}} \]
Repeating the calculations in this context, we find that
\begin{align*}
\Prob{N_k^{(1)} } &= \Prob{Z_{\rho, k}' > \left(\frac{1}{2} - \frac{1}{q}\right)(\lambda d)^k + \frac{d^k}{q} } \\
&\geq \Prob{\abs{Z_{\rho, k}' - \E{Z_{\rho, k}'}} < \left(\frac{1}{2} - \frac{1}{q} \right)(\lambda d)^k} \\
&= 1- \Prob{\abs{Z_{\rho, k}' - \E{Z_{\rho, k}'}} \geq \left(\frac{1}{2} - \frac{1}{q} \right)(\lambda d)^k} \\
& \geq 1 - \frac{\Var{Z_{\rho, k}'}}{\left(\frac{1}{2} - \frac{1}{q} \right)^2(\lambda d)^{2k}} \\
&> 1 - \frac{O(d^k) + \left(\sum_{i=1}^q \Delta_{i1}^2\right)Rd^k((\lambda^2d)^k-1)}{\left(\frac{1}{2} - \frac{1}{q} \right)^2(\lambda d)^{2k}(\lambda^2d-1)} \xrightarrow{k \rightarrow \infty} 1 - \frac{R\sum_{i=1}^q \Delta_{i1}^2}{\left(\frac{1}{2} - \frac{1}{q}\right)^2(\lambda^2d-1)}
\end{align*}
For $i \neq 1$, we have the similar computation that
\begin{align*}
\Prob{N_k^{(i)} } &= \Prob{Y_{\rho, k}^{(i)'} > \left(\frac{1}{2} - \frac{1}{q}\right)(\lambda d)^k + \frac{d^k}{q} } \\
&\leq \Prob{\abs{Y_{\rho, k}^{(i)'} - \E{Y_{\rho, k}}^{(i)'}} > \left(\frac{1}{2} - \frac{1}{q} \right)(\lambda d)^k} \\
& \leq \frac{\Var{Y_{\rho, k}^{(i)'}}}{\left(\frac{1}{2} - \frac{1}{q} \right)^2(\lambda d)^{2k}} \\
&<\frac{O(d^k) + \left(\sum_{j=1}^q \Delta_{ji}^2\right)Rd^k((\lambda^2d)^k-1)}{\left(\frac{1}{2} - \frac{1}{q} \right)^2(\lambda d)^{2k}(\lambda^2d-1)}  \xrightarrow{k \rightarrow \infty} \frac{R\sum_{j=1}^q \Delta_{ji}^2}{\left(\frac{1}{2} - \frac{1}{q}\right)^2(\lambda^2d-1)}
\end{align*}
Putting it all together, we find that 
\[ \liminf_{k \rightarrow \infty}\Prob{N_k} > 1 - \frac{R\sum_{j=1}^q \Delta_{j1}^2}{\left(\frac{1}{2} - \frac{1}{q}\right)^2(\lambda^2d-1)} - \sum_{i \neq 1} \frac{R\sum_{j=1}^q \Delta_{ji}^2}{\left(\frac{1}{2} - \frac{1}{q}\right)^2(\lambda^2d-1)} = 1 - \frac{qR\sum_{i=1}^q\sum_{j=1}^q \Delta_{ji}^2}{\left(\frac{1}{2} - \frac{1}{q}\right)^2(\lambda^2d-1)} \]
since we know that $qR \leq q-1$ from the proof of \cref{majority}, we only need to consider the terms $\sum_{i=1}^q \sum_{j=1}^q \Delta_{ij}^2$ and $\left( \frac{1}{2} - \frac{1}{q}\right)^2$. Since all of the entries in $\Delta$ are probabilities, the sum is easily bounded by $q^2$. Since $q \geq 3$, we know that $\left(\frac{1}{2} - \frac{1}{q}\right)^2 \geq \frac{1}{36}$, so all together we have that 
\[ \liminf_{k \rightarrow \infty} \Prob{N_k} > 1 - \frac{36q^2(q-1)}{\lambda^2d-1} \]
which is of the desired form. 
\end{proof}

Taking $\lambda^2d$ to be large, these probabilities get arbitrarily close to 1. Since our $E_k$ is the probability of the optimal estimator guessing the root correctly, it must also succeed with at least this probability. If we are slightly more careful with our guessing method, we can achieve an exponential bound which will be useful in the next section. In particular, we conclude with the following proposition.
\begin{prop} \label{doubleMajority}
Assume that $\lambda^2d > 4C(q)$ where $C(q)$ is the max of $C$ as in Proposition 2.1 or 2.2. Then in particular, for any fixed $d \in \mathbb{N}$ there exists $k_0 = k_0(d, \lambda, q)$ so that for all $k \geq k_0$,
\[ E_k, \tilde E_k \geq 1 - C'(q)e^{-\frac{1}{50}\lambda^2d} \]
\end{prop}
\begin{proof}
	We use the above majority estimate to provide a tight bound on the optimal probability of correctly guessing the root label as 1. The idea of this algorithm is to guess the labels on each of the roots children optimally. Then, using the majority on these children, guess the label of the root. We have the lower bound on the optimal guessing from above, and we can then compute the probability of each child being guessed as a particular label.
	
	Immediately, we know that \[ \liminf_{k \rightarrow \infty} E_k \geq 1 - \frac{C(q)}{\lambda^2d} \]
	In particular, we can find a level $k_0$ so that for all $k > k_0$, 
	\[ E_k \geq 1 - \frac{2C(q)}{\lambda^2d} \]
	Consider now $k \geq k_0+1$. Then, conditioned on the root being 1, we have that each of the $d$ subtrees rooted at its $d$ children are independent, and thus can be correctly reconstructed with probability at least $1-\frac{2C(q)}{\lambda^2d}$. In particular, the probability of guessing correctly is exactly $1-\epsilon$ for some $\epsilon \leq \frac{2C(q)}{\lambda^2d}$. We now calculate the probability of each child being guessed as either $1$ or $i \neq 1$. Denote $\hat\sigma$ to be the guess and $\sigma$ to be the true labels. 
	\begin{align*}
	\mathbb{P}(\hat\sigma(u_i) = 1) &= \mathbb{P}(\hat\sigma(u_i) = 1 | \sigma(u_i) = 1) \mathbb{P}(\sigma(u_i)=1) + \mathbb{P}(\hat\sigma(u_i) = 1|\sigma(u_i) \neq 1)\mathbb{P}(\sigma(u_i) \neq1) \\
	&= \mathbb{P}(\hat\sigma(u_i) = \sigma(u_i)) (1-p) + \frac{1-\mathbb{P}(\hat\sigma(u_i) = \sigma(u_i))}{q-1}\cdot p \\
	&= (1-\epsilon)(1-p) + \frac{\epsilon p}{q-1}
	\end{align*}
	Similarly, by using the law of total probability we can express
	\[ \mathbb{P}(\hat\sigma(u_i) = i) = (1-\epsilon)\cdot\frac{p}{q-1} + \left(1-\frac{p}{q-1}\right)\cdot \frac{\epsilon}{q-1} \]
	Let $N_i$ denote the number of guesses of each label on the children of the root, more formally $N_i = \sum_{u \in L_1(\rho)} \mathbbm{1}(\hat\sigma(u) = i)$. By the above, since each child is independent of the others, we can see that $N_i \sim \text{Bin}(d, p_i)$ where 
	\[ p_i = \begin{cases}
	(1-\epsilon)(1-p) + \frac{\epsilon p}{q-1} & i = 1 \\
	 (1-\epsilon)\cdot\frac{p}{q-1} + \left(1-\frac{p}{q-1}\right)\cdot \frac{\epsilon}{q-1} & i \neq 1
	\end{cases} \]
	Applying Hoeffding's inequality, we know that
	\[ \Prob{\abs{N_i - dp_i} > \alpha\lambda d} = \Prob{\abs{N_i - \mathbb{E}[N_i]} > \alpha\lambda d} < 2e^{-2(\alpha\lambda d)^2 / d} = 2e^{-2\alpha^2 \lambda^2d} \] 
	Applying the union bound, we get that 
	\[ \Prob{\bigcup_{i=1}^q \left\lbrace\abs{N_i - dp_i} > \alpha\lambda d\right\rbrace } \leq \sum_{i=1}^q \Prob{\abs{N_i - dp_i} > \alpha\lambda d} < \sum_{i=1}^q 2e^{-2\alpha^2\lambda^2 d} = 2qe^{-2\alpha^2\lambda^2d}  \]
	This shows that for any constant $\alpha > 0$, we know that all of the counts $N_i$ are close to their mean with probability at least $1-2qe^{-2\alpha^2\lambda^2 d}$. We would now like to determine an appropriate value for $\alpha$. We want that even with these discrepancies, that 1 still be the most common guessed label. In particular, this means for any $i$, we need that
	\begin{align*}
	(1-\epsilon)(1-p) + \frac{\epsilon p}{q-1} - (1-\epsilon)\frac{p}{q-1} - \left(1-\frac{p}{q-1}\right)\frac{\epsilon}{q-1} &> 2\alpha\lambda \\
	 (1-\epsilon)\left(1-p-\frac{p}{q-1} \right) - \frac{\epsilon}{q-1}\left(1-p-\frac{p}{q-1}\right) &> 2\alpha\lambda \\
	 (1-\epsilon)\lambda - \frac{\epsilon\lambda}{q-1} &> 2\alpha\lambda \\
	 1-\epsilon(1+\frac{1}{q-1}) &> 2\alpha \\
	 1 - \epsilon \cdot \frac{q}{q-1} &> 2\alpha
	\end{align*}
	From here, it  suffices for
	\[ 	 1 - \frac{2C(q)}{\lambda^2d}\frac{q}{q-1} > 2\alpha \]
	since we know that $\epsilon$ is bounded above by the replaced value. Notice that under the condition $\lambda^2d > 4C(q)$, it in fact suffices for 
	\[ 1-\frac{1}{2} \cdot \frac{q}{q-1} > 2\alpha \]
	to hold. Since $q \geq 3$, we know that $\frac{q}{q-1} \leq \frac{3}{2}$ so it suffices for 
	\[  2\alpha < 1- \frac{1}{2} \cdot \frac{3}{2} \iff \alpha < \frac{1}{8}\]
	Choosing $\alpha = \frac{1}{10}$ to satisfy this, we can conclude that 
	\[ \Prob{N_1 > N_i \ \forall i \neq 1} \geq \Prob{\bigcup_{i=1}^q \left\lbrace\abs{N_i - dp_i} > \frac{1}{10}\lambda d\right\rbrace } \geq 1- 2qe^{-\frac{1}{50}\lambda^2d} \]
	By the exact same proof, we have the inequality for $\tilde E_k$ as well. 
\end{proof}

\section{Contraction of Noisy and Non-noisy Estimates}
In this section we prove the contraction between $X_\rho$ and $W_\rho$ in expectation. The structure will go as follows. We first take advantage of the structure of the tree to write a Bayesian recursive formula for each of $X_\rho$ and $W_\rho$. Then, we show that with high probability, the vectors $X_\rho$ and $W_\rho$ take on a nice form for the recursion. In particular, in this setting, we have that the partial derivatives of the recursion are quite small. Finally, since we have small partial derivatives with high enough probability, we can guarantee a contraction of the random vectors. This will give as an immediate corollary the desired result for this paper. 

\subsection{The Recursive Formula}
We first derive the recursive formula for $X_\rho^{(m)}$ that is naturally induced by the recursive structure of the regular tree. This formula will be essential to our analysis of the contracting property of these random vectors. The computation that follows is an application of Bayes' rule multiple times, and follows as in an analogous computation in \cite{potts}. Suppose that we have labels $L$ at level $m$ of children.

\begin{align*}
X_\rho^{(m)}(i) &= \Prob{\sigma_\rho = i \middle\vert \sigma_{L_{m}(\rho)} = L} = \frac{\Prob{\sigma_{L_m(\rho)} = L \middle\vert \sigma_\rho = i}\Prob{\sigma_\rho = i}}{\sum_{l=1}^q \Prob{\sigma_{L_m(\rho)} = L \middle\vert \sigma_\rho = l}\Prob{\sigma_\rho = l}} \\
&= \frac{\prod_{j=1}^d \sum_{k=1}^q \Prob{\sigma_{L_m(u_j) = L\vert_{u_j}} \middle \vert \sigma_{u_j} = k} \Prob{\sigma_{u_j} = k \middle\vert \sigma_\rho = i}}{\sum_{l=1}^q \prod_{j=1}^d \sum_{k=1}^q \Prob{\sigma_{L_m(u_j) = L\vert_{u_j}} \middle \vert \sigma_{u_j} = k} \Prob{\sigma_{u_j} = k \middle\vert \sigma_\rho = l}} \\
&= \frac{\prod_{j=1}^d \sum_{k\neq i} \Prob{\sigma_{u_j} = k \middle\vert \sigma_{L_m(u_j)} = L\vert_{u_j}} \cdot \frac{p}{q-1} + \Prob{\sigma_{u_j} = i \middle\vert \sigma_{L_m(u_j)} = L\vert_{u_j}}\cdot (1-p)}{\sum_{l=1}^q \prod_{j=1}^d \sum_{k\neq l} \Prob{\sigma_{u_j} = k \middle\vert \sigma_{L_m(u_j)} = L\vert_{u_j}}\cdot \frac{p}{q-1} + \Prob{\sigma_{u_j} = l \middle\vert \sigma_{L_m(u_j)} = L\vert_{u_j}}\cdot (1-p)}\\
&= \frac{\prod_{j=1}^d \left[ \left(1-\Prob{\sigma_{u_j} = i \middle\vert \sigma_{L_m(u_j)} = L\vert_{u_j}} \right) \cdot \frac{p}{q-1} + \Prob{\sigma_{u_j} = i \middle\vert \sigma_{L_m(u_j)} = L\vert_{u_j}} \cdot (1-p) \right]}{\sum_{l=1}^q \prod_{j=1}^d \left[ \left(1-\Prob{\sigma_{u_j} = l \middle\vert \sigma_{L_m(u_j)} = L\vert_{u_j}} \right) \cdot \frac{p}{q-1} + \Prob{\sigma_{u_j} = l \middle\vert \sigma_{L_m(u_j)} = L\vert_{u_j}} \cdot (1-p) \right]} \\
&= \frac{\prod_{j=1}^d \left[ \frac{p}{q-1} + \Prob{\sigma_{u_j} = i \middle\vert \sigma_{L_m(u_j)} = L\vert_{u_j}} \cdot \left(1 - p - \frac{p}{q-1}\right) \right]}{\sum_{l=1}^q \prod_{j=1}^d \left[ \frac{p}{q-1} + \Prob{\sigma_{u_j} = l \middle\vert \sigma_{L_m(u_j)} = L\vert_{u_j}} \cdot \left(1 - p - \frac{p}{q-1}\right) \right]} \\
&= \frac{\prod_{j=1}^d \left[ \frac{1-\lambda}{q} + \Prob{\sigma_{u_j} = i \middle\vert \sigma_{L_m(u_j)} = L\vert_{u_j}} \cdot \lambda \right]}{\sum_{l=1}^q \prod_{j=1}^d \left[ \frac{1-\lambda}{q} + \Prob{\sigma_{u_j} = l \middle\vert \sigma_{L_m(u_j)} = L\vert_{u_j}} \cdot \lambda \right]} \\
&= \frac{\prod_{j=1}^d \left[ 1 + \lambda q\left( X_{u_j}^{(m-1)}(i) - \frac{1}{q}\right)\right]  }{\sum_{l=1}^q \prod_{j=1}^d \left[ 1 + \lambda q\left( X_{u_j}^{(m-1)}(l) - \frac{1}{q}\right)\right]}
\end{align*}

With the recursion derived above in mind, we define the vector valued function that captures its action. 

\begin{defn}\label{recursivefn}
Let $f: \mathbb{R}^{d \times q} \rightarrow \mathbb{R}^q$ be defined by 
\[ f_j(x_1, x_2, \ldots, x_d) = \frac{\prod_{i=1}^d (1 + \lambda q(x_i(j) - \frac{1}{q}))}{\sum_{k=1}^q \prod_{i=1}^d (1 + \lambda q(x_i(k) - \frac{1}{q}))} = 1- \frac{\sum_{k\neq j}\prod_{i=1}^d (1 + \lambda q(x_i(k) - \frac{1}{q}))}{\sum_{k=1}^q \prod_{i=1}^d (1 + \lambda q(x_i(k) - \frac{1}{q}))} \]
where $f = (f_1, f_2, \ldots, f_q) $. 
\end{defn}
With this, we have that
\[ X_\rho^{(m)} = f(X_{u_1}^{(m-1)}, \ldots X_{u_d}^{(m-1)}) \]
Also, for ease of notation we make the following definition.
\begin{notn}
	Denote $p_i$ to be the probability that a given child of the root has label $i$, so in our case
	\[ p_i  = \begin{cases}
	1-p & i=1 \\
	\frac{p}{q-1} & i \neq 1
	\end{cases}\]
\end{notn}
We wish to analyze the gradient of this function, as this will allow us to bound the differences of the output, namely $X_\rho$ and $W_\rho$ by the differences of the inputs, namely the $X_u$ and $W_u$. The second expression is included in the definition as it makes evaluating the gradient with respect to some $x_i(j)$ simpler.

The partial derivatives with respect to each input entry can be computed to be as follows
\[ \dfrac{\partial f_j}{\partial x_t(l)} = \begin{cases}
- \frac{\left( \prod_{i=1}^d (1+\lambda q(x_i(j) - \frac{1}{q})) \right) \left( \prod_{i \neq t} (1+\lambda q(x_i(l) - \frac{1}{q})) \right) \cdot \lambda q}{\left(\sum_{k=1}^q \prod_{i=1}^d (1 + \lambda q(x_j(k) - \frac{1}{q}))\right)^2} & l \neq j \\
\frac{\left( \sum_{k \neq j}\prod_{i=1}^d (1+\lambda q(x_i(k) - \frac{1}{q})) \right) \left( \prod_{i \neq t} (1+\lambda q(x_i(l) - \frac{1}{q})) \right) \cdot \lambda q}{\left(\sum_{k=1}^q \prod_{i=1}^d (1 + \lambda q(x_j(k) - \frac{1}{q}))\right)^2} & l = j
\end{cases} \]

Let $N_k = \prod_{i=1}^d (1+\lambda q(x_i(k) - \frac{1}{q}))$ and $N_{k, -t} = \prod_{i \neq t} (1+\lambda q(x_i(k) - \frac{1}{q}))$, so we can rewrite the above as 
\[ \frac{\partial f_j}{\partial x_t(l)} = \begin{cases}
- N_j \cdot \left( \sum_{k=1}^q  N_k \right)^{-2} \cdot N_{l, -t} \cdot \lambda q & l \neq j \\
\left( \sum_{k \neq j} N_k \right) \cdot \left( \sum_{k=1}^q  N_k \right)^{-2} \cdot N_{l, -t} \cdot \lambda q & l = j
\end{cases} \]
From this, we can derive an easy bound that $\abs{\frac{\partial f_j}{\partial v_t(l)}} \leq \frac{q}{1-\lambda}$, as each of the two terms involving $N_k$ in the numerator are at most the term that is squared in the denominator. The $1-\lambda$ comes from the missing factor in $N_{l, -t}$ which can be as small as exactly $1-\lambda$. Before proceeding to the full analysis, we provide a series of lemmas that will help decompose the quantity we are interested in into well-understood subcases. 

\subsection{Event of  Favorable Inputs}

First, we would like to establish that with high probability, the input vectors $x_1, \ldots, x_d$ take on a particularly nice form that will allow us to produce a strong bound. 

\begin{lemma}\label{favorableChildren}
	Let $D_i$ be the number of labels in the children of $\rho$ who are in community $i$, and $p_i$ be the probability that a fixed child has label $i$. Then
	\[ \Prob{\bigcap_{i=1}^q \left\lbrace\abs{D_i -dp_i} \leq \sqrt{d \log d}\right\rbrace} \geq 1- \frac{2q}{d^2} \]
\end{lemma}
\begin{proof}
	First, note that $n_i \sim \text{Bin}(d, p_i)$ for any $i$. Applying Hoeffding's inequality, we know that
	\[ \Prob{\abs{D_i - dp_i} > \sqrt{d \log d}} = \Prob{\abs{D_i - \mathbb{E}[D_i]} > \sqrt{d \log d}} < 2e^{-2(d \log d) / d} = 2d^{-2} \] 
	Applying the union bound, we get that 
	\[ \Prob{\bigcup_{i=1}^q \left\lbrace\abs{D_i - dp_i} > \sqrt{d \log d}\right\rbrace} \leq \sum_{i=1}^q \Prob{\abs{D_i - dp_i} > \sqrt{d \log d}} < \sum_{i=1}^q \frac{2}{d^2} = \frac{2q}{d^2}  \]
	Taking the probability of the complement gives the desired inequality. 
\end{proof}	

\begin{lemma}\label{favorableVectors}
	Assume $\lambda^2d > C(q)$ as in Proposition 2.3, then there exists $C = C'(q)$ and $m_0 = m_0(d, \lambda)$ such that the following holds for all $m \geq m_0$. Let $B_i$ be the event that vertex $u_i$ does not satisfy $\max X_{u_i} \geq 1- Ce^{-\frac{1}{100}\lambda^2d}$, and let $B = \sum_{i=1}^d \mathbbm{1}(B_i)$. Then we have that with probability at least $1- \frac{2}{d^2}$,
	\[ B \leq C_1( q) de^{-\frac{1}{100}\lambda^2d} + \sqrt{d \log d} \]
	In particular, with high probability we have that $B = o(d)$. 
\end{lemma}
\begin{proof}
	By Proposition 2.3, we know there exists $C = C(q)$ and $m_0(d, \lambda)$ such that for all $m \geq m_0$, 
	\[ E_m = \E{\max_i X_\rho^{(m)}(i)} \geq 1 - Ce^{-\frac{1}{50}\lambda^2d} \]
	
	Moreover, since $X_{\rho}^{(m)}$ is a vector of probabilities, we know that $1 - \max_i X_{\rho}^{(m)}$ is a positive random variable with expectation at most $\epsilon = Ce^{-\frac{1}{50}\lambda^2d}$ By Markov's Inequality, we get that
	\begin{align*}
	\Prob{1- \max_i X_{\rho}^{(m)} > e^{\frac{1}{100}\lambda^2d}\epsilon} &< e^{-\frac{1}{100}\lambda^2d}  \\
	\Prob{\max_i X_{\rho}^{(m)} < 1 - e^{\frac{1}{100}\lambda^2d}\epsilon} &< e^{-\frac{1}{100}\lambda^2d}
	\end{align*}
	Recalling the definition of $\epsilon$, the event being bounded is equivalent to 
	\[ \max_i X_{\rho}^{(m)} < 1 - Ce^{-\frac{1}{100}\lambda^2d} \]
   This shows that $\Prob{B_i} < e^{-\frac{1}{100}\lambda^2d} $, so $B$ is bounded by a stochastically dominated by a binomial random variable with this probability. Applying Hoeffding's inequality, we have that
   \begin{align*}
   \Prob{B<de^{-\frac{1}{100}\lambda^2d} + \sqrt{d \log d}} > 1- \Prob{\abs{B - de^{-\frac{1}{100}\lambda^2d}} > \sqrt{d \log d}} > 1 -2d^{-2}
   \end{align*}
\end{proof}

Now, we can formally define the nice form of input vectors that will be useful in the analysis. 
\begin{defn}\label{favorableEvent}
	For any $d \in \mathbb{N}$ and all $m \geq m_0$ as in Lemma 3.2, define $A_i(d)$ to be the event that we have both 
	\begin{enumerate}
	\item $D_j \in [dp_j - \sqrt{d \log d}, dp_j + \sqrt{d \log d}] \qquad \forall j\neq i $\\
	\item $B_{-i} = \sum_{j\neq i} \mathbbm{1}(B_j) \leq de^{-\frac{1}{100}\lambda^2d} + \sqrt{d \log d}$
	\end{enumerate}
	Notice that importantly, $A_i(d)$ is independent of the variable $X_\rho^{(m)}(i)$. 
\end{defn}

Moreover, notice that the above proofs rely only on the result of Proposition 2.3 for $E_k$. Since the same applies to $\tilde E_k$, we also have the same nice form on the $W_\rho^{(m)}$ with high probability. In particular, if we define $\tilde A_i(d)$ to be the analogous event regarding $W_\rho^{(m)}$, then we know that $\tilde A_i(d)$  occurs with high probability as well. Taking a union bound, we can obtain that $A_i(d)$ and $\tilde A_i(d)$ simultaneously hold also with high probability. 

\subsection{Analysis with Favorable Inputs}

Now that we have established when the input vectors take on this nice form, it is natural to provide a strong bound on the gradient under these conditions. The following lemma gives such a bound.

\begin{lemma}\label{gradientBound}
Fix $\xi > 0$ and assume that $\lambda\leq 1-\xi$. Let $\delta > 0$ satisfy 
\[ \left(1-\frac{\lambda q - \lambda q \delta}{1+\lambda(q-1)}\right)^\lambda\left( 1 + \frac{\lambda\delta q}{1+\lambda(q-1-\delta q)} \right)^{4\frac{1+\lambda(q-1)}{q}}\left(1 + \frac{\lambda q\delta}{1-\lambda}\right)^{4\frac{(1-\lambda)(q-1)}{q}} < 1 - \nu \] 
for some $\nu = \nu( q, \lambda)$. Moreover, let $\sigma \in [q]^d$ be a vector such that for all $j \in [q]$, 
\[ \abs{\abs{ \{i: \sigma(i) = j \} } - dp_j} = o(d)  \]
and suppose that for all but $o(d)$ of the $v_i$, 
\[ v_i(k) \in \begin{cases}
[0, \delta] & k \neq \sigma(i) \\
[1-\delta, 1] & k = \sigma(i)
\end{cases} \]
Then for $d$ sufficiently large and all indices $i, j, l$, 
 \[ \abs{\frac{\partial f_j}{\partial v_i(l)}} <  C_2( q, \xi)(1-\nu)^d  \]
\end{lemma}
\begin{rem}
	Such a $\delta$ exists as when $\delta = 0$, the expression is strictly less than 1, so we can choose our $\delta > 0$ by continuity.
\end{rem}

\begin{proof}

With the given assumptions we can then write upper and lower bounds on the terms that that derivative is composed of.
\begin{align*}
N_1&> \left(1 + \lambda(q - 1 -\delta q)\right)^{d(1-p) - o(d)}\left(1-\lambda\right)^{dp + o(d)} = C_- \\
N_1 &< \left(1 + \lambda(q - 1)\right)^{d(1-p) + o(d)}\left(1-\lambda(1-\delta q)\right)^{dp - o(d)} = C_+ \\
N_i&< \left(1 + \lambda(q - 1)\right)^{dp/(q-1) +  o(d)}\left(1-\lambda(1-\delta q)\right)^{d(1-p/(q-1)) - o(d)} = \delta_+
\end{align*}
We can then use these to bound each partial derivative by
\begin{align*}
\abs{\frac{\partial f_j}{\partial v_t(l)}} &\leq  \xi^{-1} \cdot \frac{(C_+ + q\delta_+)q^2\lambda \delta_+}{C_-^2} = \xi^{-1}\lambda q^2\frac{\delta_+C_+}{C_-^2} + \xi^{-1}\lambda q^3 \left(\frac{\delta_+}{C_-}\right)^2 \\
&= \xi^{-1}\lambda q^2 \frac{\delta_+}{C_+} \cdot \left(\frac{C_+}{C_-}\right)^2 + \xi^{-1}\lambda q^3 \left(\frac{\delta_+}{C_+}\right)^2\left(\frac{C_+}{C_-}\right)^2 
\end{align*}
Here, we need a term of $\xi$ as the $N_{l, -t}$ in the numerator is missing one of its multiplicative factors. This factor can be as small as $1-\lambda$, which we have bounded below by $\xi$ by assumption. Dividing by this lower bound allows us to multiply this term back into the numerator while maintaining an upper bound on the entire partial derivative. 

We first estimate $\frac{\delta_+}{C_+}$.
\begin{align*}
\frac{\delta_+}{C_+} &= \left(1 + \lambda(q - 1)\right)^{d(p + p_/(q-1) - 1)}\left(1-\lambda(1-\delta q)\right)^{d(1-p/(q-1)-p)} \\
&= \left(\frac{1-\lambda(1-\delta q)}{1+\lambda(q-1)}\right)^{\lambda d} \\
&= \left( 1 - \frac{\lambda q - \lambda q \delta}{1+\lambda(q-1)} \right)^{\lambda d}
\end{align*}
Now we estimate $\frac{C_+}{C_-}$.
\begin{align*}
\frac{C_+}{C_-} &= \frac{\left(1 + \lambda(q - 1)\right)^{d(1-p) +o(d)}\left(1-\lambda(1-\delta q)\right)^{dp - o(d)} }{\left(1 + \lambda(q - 1 -\delta q)\right)^{d(1-p) - o(d)}\left(1-\lambda\right)^{dp + o(d)} } \\
&= \left(\frac{1+\lambda(q-1)}{1+\lambda(q-1-\delta q)} \right)^{d(1-p)} \left(\frac{1-\lambda(1-\delta q)}{1-\lambda}\right)^{dp} \left(\frac{(1+\lambda(q-1))(1+\lambda(q-1-\delta q))}{(1-\lambda(1-\delta q))(1-\lambda)} \right)^{o(d)} \\
&\leq \left(\frac{1+\lambda(q-1)}{1+\lambda(q-1-\delta q)} \right)^{2d(1-p)} \left(\frac{1-\lambda(1-\delta q)}{1-\lambda}\right)^{2dp}\\
&= \left( 1 + \frac{\lambda\delta q}{1+\lambda(q-1-\delta q)} \right)^{2d(1-p)}\left(1 + \frac{\lambda q\delta}{1-\lambda}\right)^{2dp}
\end{align*}
for $d$ sufficiently large. Here, the third equality follows from the observation that as the exponent is sub-linear in $d$, the entire term will be dominated once $d$ is large enough. This will be made precise at the end of this proof. 

Recall that $p = \frac{(1-\lambda)(q-1)}{q}$ and $1-p = \frac{1+\lambda(q-1)}{q}$. Now, by selection of  $\delta $, we have that
\begin{align*}
\left(1-\frac{\lambda q - \lambda q \delta}{1+\lambda(q-1)}\right)^\lambda\left( 1 + \frac{\lambda\delta q}{1+\lambda(q-1-\delta q)} \right)^{4\frac{1+\lambda(q-1)}{q}}\left(1 + \frac{\lambda q\delta}{1-\lambda}\right)^{4\frac{(1-\lambda)(q-1)}{q}} < 1 - \nu
\end{align*}
Putting these together, we find that 
\begin{align*}
\abs{\frac{\partial f_j}{\partial v_t(l)}} &\leq  \xi^{-1}\lambda q^2\frac{\delta_+}{C_+} \cdot \left(\frac{C_+}{C_-}\right)^2 + \xi^{-1}\lambda q^3 \left(\frac{\delta_+}{C_+}\right)^2\left(\frac{C_+}{C_-}\right)^2 \\
&< \frac{\lambda q^2 + \lambda q^3}{\xi} \cdot \frac{\delta_+}{C_+} \cdot \left(\frac{C_+}{C_-}\right)^2 \\
&< \frac{\lambda q^2 + \lambda q^3}{\xi} \cdot  \left( 1 - \frac{\lambda q - \lambda q \delta}{1+\lambda(q-1)} \right)^{\lambda d} \cdot \left(\left( 1 + \frac{\lambda\delta q}{1+\lambda(q-1-\delta q)} \right)^{2d(1-p)}\left(1 + \frac{\lambda q\delta}{1-\lambda}\right)^{2dp}\right)^2 \\
&= \lambda C(q, \xi)\left( 1 - \frac{\lambda q - \lambda q \delta}{1+\lambda(q-1)} \right)^{\lambda d} \cdot \left( 1 + \frac{\lambda\delta q}{1+\lambda(q-1-\delta q)} \right)^{4d(1-p)}\left(1 + \frac{\lambda q\delta}{1-\lambda}\right)^{4dp}  \\
&< C(q, \xi) (1-\nu)^d
\end{align*}
So, this gives the desired result, once we verify the bound as $d$ gets large. Notice that 
\[ 1 = \left(\frac{(1+\lambda(q-1))(1+\lambda(q-1-\delta q))}{(1-\lambda(1-\delta q))(1-\lambda)} \right)^{0} < \left( 1 + \frac{\lambda\delta q}{1+\lambda(q-1-\delta q)} \right)^{1-p}\left(1 + \frac{\lambda q\delta}{1-\lambda}\right)^{p} \]
Thus, we can find $\epsilon > 0$ small enough so that 
\[ \left(\frac{(1+\lambda(q-1))(1+\lambda(q-1-\delta q))}{(1-\lambda(1-\delta q))(1-\lambda)} \right)^{\epsilon} < \left( 1 + \frac{\lambda\delta q}{1+\lambda(q-1-\delta q)} \right)^{1-p}\left(1 + \frac{\lambda q\delta}{1-\lambda}\right)^{p} \]
Then, taking $d$ large enough so that $o(d) < \epsilon d$ gives that 
\[ \left(\frac{(1+\lambda(q-1))(1+\lambda(q-1-\delta q))}{(1-\lambda(1-\delta q))(1-\lambda)} \right)^{o(d)} < \left( 1 + \frac{\lambda\delta q}{1+\lambda(q-1-\delta q)} \right)^{d(1-p)}\left(1 + \frac{\lambda q\delta}{1-\lambda}\right)^{dp} \]
We will provide precise analysis for the behavior of $\delta, \epsilon$ and $d$ below in Lemma 3.5.
\end{proof}

We now present the promised two lemmas that provide a more detailed analysis on the specific behavior of the $\delta$, $\nu$ and $\epsilon$ in question above. This will aid in determining the dependence of $d$ on $\lambda$ in the final calculation. We will use the bound that $1+x \leq e^x$ for any number $x$. 

\begin{lemma}\label{delta}
	The conditions to Lemma 3.3 are satisfied for $\delta = C_\delta(q) \lambda$ and $\nu = C_\nu(q) \lambda^2$.
\end{lemma}
\begin{proof}
By the stated inequality, we in order for 
\[\left(1-\frac{\lambda q - \lambda q \delta}{1+\lambda(q-1)}\right)^\lambda\left( 1 + \frac{\lambda\delta q}{1+\lambda(q-1-\delta q)} \right)^{4\frac{1+\lambda(q-1)}{q}}\left(1 + \frac{\lambda q\delta}{1-\lambda}\right)^{4\frac{(1-\lambda)(q-1)}{q}} < 1\]
to hold, if suffices to ensure that
\begin{align*}
& && \exp( -\frac{\lambda q - \lambda q \delta}{1 + \lambda(q-1)} )^\lambda\exp(\frac{\lambda\delta q}{1+\lambda(q-1-\delta q)})^{4\frac{1+\lambda(q-1)}{q} } \exp(\frac{\lambda \delta q}{1-\lambda})^{4 \frac{(1-\lambda)(q-1)}{q}} < 1 \\
\iff& && \exp(-\frac{\lambda^2q - \lambda^2 q\delta}{1+\lambda(q-1)} + 4\cdot \frac{1+\lambda(q-1)}{q} \cdot \frac{\lambda \delta q}{1+\lambda(q-1-\delta q)} + 4\cdot \frac{(1-\lambda)(q-1)}{q} \cdot \frac{\lambda q \delta}{1-\lambda}) < 1 \\
\iff & &&-\frac{\lambda^2q - \lambda^2 q\delta}{1+\lambda(q-1)} + 4\cdot \frac{1+\lambda(q-1)}{q} \cdot \frac{\lambda \delta q}{1+\lambda(q-1-\delta q)} + 4\cdot \frac{(1-\lambda)(q-1)}{q} \cdot \frac{\lambda q \delta}{1-\lambda} < 0 \\
\iff& && 4\cdot \frac{1+\lambda(q-1)}{1+\lambda(q-1) - \lambda\delta q} \cdot \lambda \delta + 4(q-1)\lambda\delta < \frac{\lambda^2q -\lambda^2q\delta}{1+\lambda(q-1)}
\end{align*}
First, consider the coefficient $\frac{1+\lambda(q-1)}{1+\lambda(q-1)-\lambda\delta q}$ in the left hand side. This can be written as 
\[ 1 + \frac{\lambda \delta q}{1+\lambda(q-1) - \lambda \delta q} < 1 + \lambda \delta q \]
Moreover, in the right hand side, we can bound the denominator by $1  +\lambda(q-1) \leq q$, so overall if suffices to ensure that
\begin{align*}
& 4 \cdot (1 + \lambda\delta q) \lambda \delta + 4(q-1)\lambda\delta < \lambda^2 - \lambda^2\delta \\
\iff\qquad  &4\lambda \delta + 4\lambda^2\delta^2 q + 4(q-1)\lambda\delta < \lambda^2-\lambda^2\delta \\
\iff \qquad& 4q\delta + 4\lambda q\delta^2 < \lambda - \lambda\delta
\end{align*}
Since we may assume $\delta < 1$, we have that $\delta^2 < \delta$, so it then suffices to have
\begin{align*}
& (4q+4\lambda q)\delta < \lambda-\lambda\delta \\
\iff \qquad &(4q+4\lambda q+\lambda)\delta < \lambda \\
\iff \qquad &\delta < \frac{\lambda}{4q+4q\lambda + \lambda}
\end{align*}
We can upper bound the denominator by $8q+1$, so in order to satisfy this inequality we only need that $\delta < \frac{1}{8q+1} \cdot \lambda$. Taking $\delta = \frac{1}{16q}\cdot \lambda$ gives the form required by the lemma. 

Now, we would like to figure out exactly how far away from 1 the right hand side of the original expression is. Recall the intermediate inequality $4q\delta + 4\lambda q \delta^2 < \lambda - \lambda \delta$. Through the above manipulations, we know that our original expression is bounded above by $\exp(4q\delta + 4\lambda q \delta^2 - \lambda +\lambda \delta)$. Thus, in order to conclude, we require a statement of the form 
\[ 4q\delta + 4\lambda q \delta^2 - \lambda +\lambda \delta < \log(1 - C_\nu(q) \lambda^2) \]
Using the bound $-x \leq \log(1-\frac{1}{2}\cdot x)$ for $x \in [0, 1]$, it suffices to show that
\[ 4q\delta + 4\lambda q \delta^2 - \lambda +\lambda \delta < -2C_\nu(q)\lambda^2 \]
provided that $C_\nu(q) \in [0, \frac{1}{2}]$. Substituting in the above formula for $\delta$, we obtain
\begin{align*}
& \frac{1}{4}\cdot \lambda + \frac{1}{64q} \cdot \lambda^3 - \lambda + \frac{1}{16q}\cdot \lambda^2 < - 2C_\nu(q)\lambda^2 \\
\iff \qquad & -\frac{3}{4} + \frac{1}{16q} \cdot \lambda + \frac{1}{64q} \lambda^2 < -2C_\nu(q)\lambda
\end{align*}
Since we know that $\lambda \in [0, 1]$, we can upper bound $\lambda^2$ by $\lambda$, so we only need that
\[ -\frac{3}{4} + \frac{5}{64q}\lambda < -2C_\nu(q) \lambda  \iff \left(\frac{5}{64q} + 2C_\nu(q)\right)\lambda < \frac{3}{4}\]
Taking $C_\nu(q) = \frac{1}{4}$ ensures that this holds for $\lambda \in [0, 1]$, so it suffices for our purpose. Tracing back through the derivation, we see that we get a bound of $1-\frac{1}{4}\lambda^2$, which is of the desired form. Notice that here, we do not even need the constant to depend on $q$. 
\end{proof}

\begin{lemma}\label{epsilon}
	Fix $\xi > 0$. Assuming $\lambda \leq 1-\xi$, to conclude Lemma 3.3, it suffices for $\epsilon = C_\epsilon(q, \xi) \lambda$ and thus, for $\lambda^2d > C_3(q, \xi)\log{d}$.
\end{lemma}
\begin{proof}
	We want to determine when 
	\begin{align*}
	\left(\frac{(1+\lambda(q-1))(1+\lambda(q-1-\delta q))}{(1-\lambda(1-\delta q))(1-\lambda)} \right)^{\epsilon} &< \left( 1 + \frac{\lambda\delta q}{1+\lambda(q-1-\delta q)} \right)^{1-p}\left(1 + \frac{\lambda q\delta}{1-\lambda}\right)^{p} \\
	\intertext{First, notice that}
	\left(\frac{(1+\lambda(q-1))(1+\lambda(q-1-\delta q))}{(1-\lambda(1-\delta q))(1-\lambda)} \right)^{\epsilon} &= \left(1 + \frac{\lambda q - \lambda a\delta}{1-\lambda+\lambda q\delta}\right)^\epsilon\left( 1 + \frac{\lambda q - \lambda q \delta}{1 - \lambda} \right)^\epsilon \\
	\intertext{so similarly to the above lemma, we have the upper bound}
	\left(\frac{(1+\lambda(q-1))(1+\lambda(q-1-\delta q))}{(1-\lambda(1-\delta q))(1-\lambda)} \right)^{\epsilon} &\leq \exp \left(\frac{\lambda q - \lambda a\delta}{1-\lambda+\lambda q\delta}\right)^\epsilon\exp\left( \frac{\lambda q - \lambda q \delta}{1 - \lambda} \right)^\epsilon \\
	\intertext{Moreover, notice that the term $(1 + \frac{\lambda q\delta}{1-\lambda})^p$ is always greater than one, so we can safely omit this term when analyzing a sufficient inequality. Finally, for $x \in [0, 1]$, we have that bound $e^{x/2} \leq 1+x$, so we can lower bound the other term in the right hand side}
	\left( 1 + \frac{\lambda\delta q}{1+\lambda(q-1-\delta q)} \right)^{1-p} &\geq \exp(\frac{\frac{1}{2}\lambda \delta q}{1+\lambda(q-1-\delta q)})^{1-p} \\
	\intertext{By this reasoning, it suffices for $\epsilon$ to satisfy the inequality}
	\exp \left(\frac{\lambda q - \lambda a\delta}{1-\lambda+\lambda q\delta}\right)^\epsilon\exp\left( \frac{\lambda q - \lambda q \delta}{1 - \lambda} \right)^\epsilon &< \exp(\frac{\frac{1}{2}\lambda \delta q}{1+\lambda(q-1-\delta q)})^{1-p} \\
	\intertext{Starting by taking logs, we can simplify this to} 
	\epsilon \left( \frac{\lambda q - \lambda \delta q}{1-\lambda + \lambda \delta q} + \frac{\lambda q - \lambda \delta q}{1-\lambda} \right) &< \frac{1+\lambda(q-1)}{q} \cdot \frac{1}{2} \cdot \frac{\lambda \delta q}{1+\lambda(q-1-\delta q)} \\
	\epsilon \cdot \frac{(\lambda q - \lambda \delta q)(2-2\lambda + \lambda \delta q)}{(1-\lambda)(1-\lambda + \lambda \delta q)} &< \frac{1}{2} \cdot \frac{1+\lambda(q-1)}{1+\lambda(q-1) - \lambda \delta q} \cdot \lambda \delta
	\end{align*}
	\begin{align*}
	\epsilon &< \frac{1}{2} \cdot \frac{(1-\lambda)(1-\lambda + \lambda \delta q)}{(\lambda q - \lambda \delta q)(2-2\lambda + \lambda \delta q)} \cdot \frac{1+\lambda(q-1)}{1+\lambda(q-1) - \lambda \delta q} \cdot \lambda \delta \\
	\intertext{Here, notice that the term $\frac{1+\lambda(q-1)}{1+\lambda(q-1)-\lambda \delta q} > 1$, so we can safely lower bound it by 1. Similarly, the term $\frac{1-\lambda + \lambda\delta q}{2-2\lambda + \lambda \delta q} > \frac{1}{2}$ so we can use this as a lower bound as well. Canceling the $\lambda$ in $\lambda q - \lambda\delta q$ with the $\lambda$ in the last term, we find that it suffices for }
	\epsilon &< \frac{1}{4} \cdot \frac{1-\lambda}{q-q\delta} \cdot \delta \\
	\intertext{Substituting in the expression for $\delta$ from \cref{delta}, we get that}
	\epsilon &< \frac{1}{4} \cdot \frac{1-\lambda}{q - \frac{1}{16}\lambda} \cdot \frac{\lambda}{16q} \\
	\intertext{Recall that we have the assumption $\lambda \leq 1-\xi$, so $1-\lambda$ is bounded below by $\xi$. Since $\frac{1}{q-\frac{1}{16}\lambda}$ can be lower bounded by $\frac{1}{q}$, we find that it suffices for} 
	\epsilon &< \frac{\xi}{64q^2} \cdot \lambda \\
	\intertext{This gives an expression for $\epsilon$ of the desired form by taking} 
	\epsilon &= \frac{\xi}{128q^2} \cdot \lambda 
	\end{align*}
	Now, in order for the $o(d)$ term to be at most $\epsilon d$, we need that
	\begin{align*}
	de^{-\frac{1}{100}\lambda^2d} + 2\sqrt{d \log d} &< C_\epsilon(q, \xi)\lambda d \\
	e^{-\frac{1}{100}\lambda^2d} + 2\sqrt\frac{\log d}{d} &< C_\epsilon(q, \xi) \lambda
	\end{align*}
	We consider each of these terms separately. First, 
	\begin{gather*}
	e^{-\frac{1}{100}\lambda^2d} < \frac{C_\epsilon(q, \xi)}{2} \lambda  \iff -\frac{1}{100}\lambda^2d < C'(q, \xi) + \log(\lambda) \iff \lambda^2d > -C''(q, \xi)\log(\lambda)
	\end{gather*}
	For the second term,
	\[ 2 \sqrt\frac{\log d}{d} < \frac{C_\epsilon(q, \xi)}{2} \lambda \iff \frac{4}{C_\epsilon(q, \xi)}\sqrt{\log d} < \lambda \sqrt{d} \iff \frac{16}{C_\epsilon(q, \xi)^2}\log d < \lambda^2d \]
	I claim that if suffices for $\lambda^2 d \geq C(q, \xi) \log d$ in order for both of these to hold. The second is automatically true as long as $C(q, \xi) > \frac{16}{C_\epsilon(q, \xi)^2}$, so we check the first. Our claimed condition is equivalent to $\frac{1}{\lambda} \leq \sqrt{\frac{d}{C(q, \xi)\log d}}$. Thus, in order to satisfy the first inequality it suffices for 
	\begin{align*}
	\lambda^2d &> C''(q, \xi)\log(\sqrt{\frac{d}{C(q, \xi)\log d}}) = \frac{C''(q, \xi)}{2}(\log(d) - \log(C(q, \xi)\log(d))) \\
	&= \left(\frac{C''(q, \xi)}{2} - o(1)\right)\log(d) 
	\end{align*}
	In particular, it suffices for 
	\[ \lambda^2d > \frac{C''(q, \xi)}{2} \log d \]
	Notice that $d \geq \lambda^2d > C\log d$ implies that $\frac{d}{\log d} > C \implies d > C_0$ for some constant $C_0$. Thus, taking $C(q, \xi)$ large enough, we can guarantee that all of these conditions hold, which is the desired result. 
\end{proof}

\begin{rem}
	Throughout this section, we have not carefully analyzed the fact that one of the terms in the partial derivative, namely the $N_{l, -t}$ is missing a multiplicative term that could be less than 1, in particular as small as $1-\lambda$. There are two possible resolutions to this. The first, which we use, is that $\lambda$ is bounded away from 1,  which gives us the result for most common situations. The other is that for this term to be so small, we require that its corresponding input entry $x_i(k)$ to also be small. In particular, the contribution to the norm from this particular entry will be \[ x_i(k) \cdot \frac{C}{1+\lambda(qx_i(k) - 1)} \xrightarrow{\lambda \rightarrow 1} \frac{C}{q} \]
	Notice that this is independent of $x_i(k)$, so as a result the actual contributions of these entries to the change in the function does not blow up. Thus, if we are more careful, our analysis should hold up to some constant depending on $q$, which does not affect the asymptotic behavior in $d$ that we are interested in.
\end{rem}

\subsection{Contraction Property}

Finally, we may combine the results of the above lemmas and show the contraction at large enough levels of the tree. The main idea of the proof is to expand the given vector into a telescoping sum. Then, in each difference exactly one input is changing, and we can then bound this difference conditioned on properties about these other fixed inputs. 

\begin{prop}\label{contraction}
	Fix $\xi > 0$. Given $\lambda \leq 1-\xi$ and $ q$, there exists $d_0(\lambda, q, \xi)$ such that for all $d \geq d_0$ and all $m \geq m_0(d)$,
	 \[ \E{\norm{X_\rho^{(m)} - W_\rho^{(m)}}_1} < \frac{1}{2} \cdot\E{\norm{X_\rho^{(m-1)} - W_\rho^{(m-1)}}_1} \]
	 In particular, there exists some constant $C^*(q, \xi)$ so that $\lambda^2d > C^*(q, \xi)\log{d}$ suffices for the above to hold.
\end{prop}

\begin{proof}
Recall that we define $X_\rho^{(m)}$ and $W_\rho^{(m)}$ in terms of the vectors at their children $X_u^{(m-1)}$ and $W_u^{(m-1)}$. Moreover, the distribution of $X_u^{(m-1)}$ is the same a permutation of the distribution of $X_\rho^{(m-1)}$ conditioned on the label $\sigma_u = i$. In particular, $X_u^{(m-1)}(i)$ is distributed like $X_\rho^{(m-1)}(\pi(i))$ for any $\pi \in S_q$ such that $\pi(q) = \sigma_u$. With this, we input vectors $v_i = X_{u_i}^{(m-1)}$ and $v_i' = W_{u_i}^{(m-1)}$ into our function and first simplify. We consider the $L^1$ norm, and expand the desired difference in the following way.
\begin{align*}
\norm{f(v_1, \ldots, v_d) - f(v_1', \ldots, v_d') }_1 &= \sum_{j=1}^q \abs{ f_j(v_1, \ldots, v_d) - f_j(v_1', \ldots, v_d')} \\
\intertext{Here, we expand the sum by telescoping so that each term has only one input vector changing.} 
&\leq \sum_{j=1}^q \sum_{i=0}^{d-1} \abs{ f_j(v_1', \ldots, v_i', v_{i+1}, \ldots, v_d) - f_j(v_1', \ldots, v_{i+1}', v_{i+2}, \ldots, v_d)} \\
\intertext{Then for each of these differences, we may bound it by the difference in norm between the two entries, times the gradient of the function $f_j$ with respect to that entry. }
&\leq \sum_{j=1}^q \sum_{i=1}^d \sum_{k=1}^q \abs{v_i(k) - v_i'(k)} \abs{\frac{\partial f_j}{\partial x_i(k)}} \\
\intertext{The term in the sum above is bounded by the difference in $L^1$-norm of the two vectors, multiplied by the maximum entry in the gradient, so}
&\leq \sum_{j=1}^q \sum_{i=1}^d \norm{v_i - v_i'}_1 \norm{\frac{\partial f_j}{\partial x_i}}_\infty
\end{align*}
Taking expectations, we find that 
\[ \E{\norm{f(v_1, \ldots, v_d) - f(v_1', \ldots, v_d') } } \leq \sum_{j=1}^q\sum_{i=1}^d \E{\norm{v_i - v_i'}_1\norm{\frac{\partial f_j}{\partial v_i}}_\infty} \]
Finally, expanding the expectation by conditioning on the events $A_i$ and $\tilde A_i^c$, we get an upper bound of 
\[ \sum_{j=1}^q\sum_{i=1}^d \E{\norm{v_i - v_i'}_1 \norm{\frac{\partial f_j}{\partial v_i}}_\infty \middle\vert A_i, \tilde A_i} \Prob{A_i, \tilde A_i}  + \E{\norm{v_i - v_i'}_1\norm{\frac{\partial f_j}{\partial v_i}}_\infty \middle\vert A_i^c \lor \tilde A_i^c} \Prob{A_i^c \lor \tilde A_i^c} \]

From here, we look to bound $\norm{\frac{\partial f}{\partial v_i}}_\infty$ in the case that $A_i(d)$ and $\tilde A_i(d)$  occur, with the goal of achieving a small bound. Recall in the case $A_i(d)^c \lor \tilde A_i(d)^c$, we can simply use the universal bound $\norm{\frac{\partial f}{\partial v_i}}_\infty < \frac{q}{1-\lambda} \leq \frac{q}{\xi}$. 

Notice that 
\[ D_j \in [dp_j - \sqrt{d \log d}, dp_j + \sqrt{d \log d}] \qquad \forall j  \]
implies condition 1 for $A_i(d)$. Similarly, 
\[ B \leq e^{-\frac{1}{100}\lambda^2d} + \sqrt{d \log d}\]
implies condition 2 for $A_i(d)$. Here, we may safely assume that $\lambda^2d > C$ so that the assumptions to Lemma 3.2 are satisfied, as we will require that $\lambda^2d > C\log d$ later in this proof. Thus, as a result of Lemmas 3.1 and 3.2, since $A_i(d)$ is a super-set of the intersection of the two events in question, by the union bound we have that for all $i$,
\[ \mathbb{P}(A_i(d)) \geq 1 - \frac{2q+2}{d^2} \]
Then, by the note at the end of section 3.2, we have that both $A_i(d)$ and $\tilde A_i(d)$ occur with probability $\geq 1 - \frac{4q+4}{d^2}$. 

Now, let $\delta$ be as in Lemma 3.3, and fix constants so that the following conditions hold:
\begin{itemize}
	\item $d_1 = d_1(\lambda, q)$ such that $\forall d \geq d_1$, $C_q(q)e^{-\frac{1}{100}\lambda^2d} < \delta = C_\delta(q)\lambda$ 
	\item $d_2$ sufficiently large so that for all $d \geq d_2$, $\lambda^2d > C_3(q, \xi) \log{d}$ as in Lemma 3.5
\end{itemize} 
By all of the results compiled above, the conditions to Lemma $3.3 $ are satisfied so we obtain the conclusion under these bounds on $d$. Thus, we have that given $A_i(d)$ and $\tilde A_i(d)$ for $d \geq \max{d_1, d_2}$, the gradient is bounded by $ C_2(q, \xi)(1-\nu)^d$.
We can now make the final calculation bounding 
\[ \sum_{j=1}^q\sum_{i=1}^d \E{\norm{v_i - v_i'}_1 \norm{\frac{\partial f_j}{\partial v_i}}_\infty \middle\vert A_i(d)} \Prob{A_i(d)}  + \E{\norm{v_i - v_i'}_1\norm{\frac{\partial f_j}{\partial v_i}}_\infty \middle\vert A_i(d)^c} \Prob{A_i(d)^c} \]
Replacing each term with its bound described above, we get an upper bound of
\begin{align*}
&\leq \sum_{j=1}^q \sum_{i=1}^d \left(  C_2(q, \xi)(1-\nu)^d\E{\norm{v_i - v_i'}} + \frac{q}{\xi}\E{\norm{v_i - v_i'}} \cdot\frac{4q+4}{d^2} \right) \\
\intertext{Expanding the sums by multiplying by $qd$ since we have removed dependence on $i$ and $j$, }
& < q\left(d C_2(q, \xi)(1-\nu)^d  + \frac{q}{\xi}\cdot\frac{4q+4}{d} \right) \E{\norm{v - v'}} \\
\intertext{Finally, for $d \geq d_3(\lambda, q)$ large enough we have that this is bounded by }
&< \frac{1}{2} \cdot \E{\norm{v - v'}}
\end{align*}
since overall the coefficient tends to 0 as $d $ increases. Thus, for $d \geq \max\{d_1, d_2, d_3\}$ and choosing $m \geq m_0(d, \lambda)$ as in Lemma 3.2, we find that 
\[ \E{\norm{X_\rho^{(m)} - W_\rho^{(m)}}_1} \leq  \frac{1}{2} \cdot \E{\norm{X_\rho^{(m-1)} - W_u^{(m-1)}}_1 } \]

More explicitly, we can compute these $d_1$ and $d_3$. For $d_1$, we need
\begin{gather*}
C_1(q)e^{-\frac{1}{100}\lambda^2d} <  C_\delta(q) \lambda \quad \iff \quad e^{-\frac{1}{100}\lambda^2d} < \frac{C_\delta(q)}{C_1(q)}\lambda
\end{gather*}
By a similar analysis as in Lemma 3.5, we know that it suffices for $\lambda^2d > C(q) \log d$ in order for this to hold. 
For $d_3$, we need first that 
\[ \frac{1}{\xi}\cdot \frac{4q+4}{d} < \frac{1}{4} \iff d > \frac{16q^2+16q}{\xi}  \]
Moreover, we also need
\begin{gather*}
d C_2(q)(1-\nu)^d < \frac{1}{4} \iff \log d  + \log C_2(q, \xi) + d\log(1 - C_\nu(q)\lambda^2) < -\log 4 
\end{gather*}
Since we can upper bound $\log(1-x) $ by $-x$, it suffices to have that 
\begin{gather*}
 \log d  + \log C_2(q, \xi) - C_\nu(q)\lambda^2d < -\log 4 
\iff \log d + C_4(q, \xi) < C_\nu(q)\lambda^2d 
\end{gather*}

With this, it suffices for $\lambda^2d > C_5(q) \log d$. Recall that we can require $d > D$ by requiring $\lambda^2d > D'\log{d}$ for any absolute constant $D$. Thus, since we have bounds of the form $d > D$ and $\lambda^2d > C\log d$, we can find $C^*$ a constant in terms of $q$ and $\xi$ large enough so that the single requirement of $\lambda^2d > C^*\log{d}$ suffices. 
\end{proof}

\begin{rem}
	The result is expected to hold for $\lambda^2d > C$ for some absolute constant $C$. The proof of this may be achieved with an alteration of either by separating the case when $\lambda < \epsilon$ for some fixed small $\epsilon$, or with a more careful analysis of the gradient including signs, as there is expected to be a fair amount of cancellation. 
\end{rem}	

\section{Qualified Main Theorem}

Finally, we have our first main result of the paper, which follows easily from the proven propositions.

\begin{thm}\label{equalLimit}
	Fix $\xi > 0$. Suppose that $\lambda^2d > C^*(q, \xi)\log d$ and $\lambda \leq 1-\xi$ where $C^*$ is as in Proposition 3.1, then 
	\[ \lim_{m \rightarrow \infty} E_m = \lim_{m \rightarrow \infty} \tilde E_m \]
\end{thm}

\begin{proof}
	Let $m_0$ be large enough so that the conditions of Proposition 3.1 hold. Then, we have that for all $m > m_0$, \[ \E{\norm{X_\rho^{(m)} - W_\rho^{(m)}}_1} \leq \frac{1}{2^{m-m_0}} \E{\norm{X_\rho^{(m_0)} - W_\rho^{(m_0)}}_1} \]
	by induction. Then, taking the limit as $m \rightarrow \infty$, we find that 
	\[ \lim_{m \rightarrow \infty} \E{\norm{X_\rho^{(m)} - W_\rho^{(m)}}_1} = 0 \]
	In the limit, the two vectors tend to the same entries. Since our probabilities of guessing correctly $E_m$ and $\tilde E_m$ are entry-wise maximums of $X_\rho^{(m)}$ and $W_\rho^{(m)}$ respectively, we have the result that
	\[ \lim_{m \rightarrow \infty} \abs{E_m - \tilde E_m} = 0 \]
	which is what we wanted to show.
\end{proof}

\section{Simple Majority  On Random Trees}

From here, our first step is to generalize the above computation to the setting of a random tree. In particular, we work on a Galton-Watson tree where each node independently has $\text{Pois}(d)$ children. We now re-analyze the relevant random variables in the context of this setting and show that the same results hold through the simple majority calculation. 

\subsection{Non-noisy Setting}
Define the same variables as in the Junior Paper, this time with the underlying model being a Galton-Watson tree where each child has $\text{Pois}(d)$ children independently and at random. We would first like to recalculate the expectations of these random variables. 

\begin{lemma}
	\begin{align*}
	\mathbb{E}[Z_{\rho, k}] &= \left(1-\frac{1}{q}\right)\lambda^kd^k + \frac{d^k}{q} \\
	\mathbb{E}[Y_{\rho, k}^{(i)}] &= -\frac{1}{q} \cdot \lambda^k d^k+ \frac{d^k}{q} 
	\end{align*}
\end{lemma}

\begin{proof}
	We take a recursive approach to this calculation. First, note that we expect there to be $d^k$ nodes at level $k$, so we must have that 
	\[ \E{Z_{\rho, k} + \sum_{i=2}^q Y_{\rho, k}^{(i)}} = d^k \] 
	At the $k-1$st level, there are $Z_{\rho, k-1}$ nodes labeled 1 and $Y_{\rho, k-1}^{(i)}$ nodes labeled $i$ for every other community $i$. To compute the number of nodes labeled 1 in the $k$th sublevel, we consider the children of each of these nodes independently. For each child in community 1, we expect $d(1-p)$ of its children to have label 1 as well, by taking the expectation over conditioning on the number of children. Similarly, for each of the nodes in community $i \neq 1$, we expect there to be $\frac{dp}{q-1}$ children that have label 1. Thus, recursively, we can write 
	\begin{align*}
	\E{Z_{\rho, k}} &= \E{Z_{\rho, k-1}}\cdot d(1-p) + \sum_{i=2}^q \E{Y_{\rho, k-1}^{(i)}} \cdot \frac{dp}{q-1} \\
	\intertext{Since we expect there to be $d^{k-1}$ nodes on the $k-1$st sublevel, this is equivalent to }
	\E{Z_{\rho, k}} &= \E{Z_{\rho, k-1}} \cdot d(1-p) + \left( d^{k-1} - \E{Z_{\rho, k-1}} \right) \cdot \frac{dp}{q-1}
	\end{align*}
	This is precisely the same recurrence we obtained in the fixed tree regime, and so we know from the previous calculation that 
	\begin{align*}
	\mathbb{E}[Z_{\rho, k}] &= \left(1-\frac{1}{q}\right)\lambda^kd^k + \frac{d^k}{q} \\
	\mathbb{E}[Y_{\rho, k}^{(i)}] &= -\frac{1}{q} \cdot \lambda^k d^k+ \frac{d^k}{q} 
	\end{align*}
\end{proof}

Now we move to recalculating the variance of these random variables. 

\begin{lemma}
	For any $k$ and any $i$, we have that
	\[\Var{Z_{\rho, k}}, \Var{Y_{\rho, k}^{(i)}} < Rd^k \frac{(\lambda^2 d)^k - 1}{\lambda^2 d-1}\]
\end{lemma}

\begin{proof}
	We take the same approach, but note that we expect the variance in this case to be larger, as there is the additional variance in the number of children each node has in the random tree. We specify the places where this calculation differs. Starting with the first term, we have 
	\begin{align*}
	&\E{\Var{\tilde Z_{\rho, k} \middle \vert \sigma_{L_1(\rho)}}} \\
	\intertext{Expanding the expression $\tilde Z_{\rho, k}$ into a summation, we get}
	= &\ \E{\Var{\sum_{L_1(\rho)} \tilde Z_{u, k-1} \middle\vert \sigma_{L_1(\rho)}}} \\
	\intertext{Since random variables depending on distinct children are independent, we can pull the sum out of the variance to get}
	=&\  \E{\sum_{L_i(\rho)} \Var{\tilde Z_{u, k-1} \middle\vert \sigma_u}} \\
	\intertext{Now, using the Tower Rule to introduce conditioning on $D$ within the expectation, we get the expression}
	=&\  \E{D \Var{\tilde Z_{u, k-1} \middle\vert \sigma_u}} \\
	\intertext{Since $D$ and the Variance are independent, we may factor and get }
	=&\  \E{D}\E{\Var{\tilde Z_{u, k-1} \middle\vert \sigma_u}} = d((1-p)C_{k-1} + pD_{k-1})
	\end{align*}
	which is the same expression as previously. 
	
	The primary difference occurs when computing the term 
	\[ \Var{\E{\tilde Z_{\rho, k} \middle\vert \sigma_{L_1(\rho)}}} = \Var{\sum_{L_1(\rho)} \E{\tilde Z_{u, k-1} \middle\vert \sigma_u}} \]
	as we can no longer simply pull the summation outside of the variance by independence. Now, we need to condition on the number of children and expand the variance once again under this conditioning. We now have the expression
	\[ \E{\Var{\sum_{L_1(\rho)} \E{\tilde Z_{u, k-1} \middle\vert \sigma_u} \middle\vert D}} + \Var{\E{\sum_{L_1(\rho)} \E{\tilde Z_{u, k-1} \middle\vert \sigma_u} \middle\vert D}} \]
	This then becomes 
	\[ \E{D \Var{\E{\tilde Z_{u, k-1} \middle\vert \sigma_u}}} + \Var{D\E{\tilde Z_{u, k-1}}} \]
	Recall from \cref{varianceYZ} that 
	\[ \E{\tilde Z_{u, k-1} \middle\vert \sigma_u} \sim \left(\text{Ber}(1-p) - \frac{1}{q}\right) \cdot (\lambda d)^{k-1} \]
	Thus, we know that 
	\begin{align*}
	\E{\tilde Z_{u, k-1}} &= \left(1-p-\frac{1}{q}\right)\cdot (\lambda d)^{k-1} \\
	 \Var{\E{\tilde Z_{u, k-1} \middle\vert \sigma_u}} &= (\lambda d)^{2k-2}\cdot p(1-p)
	\end{align*}
	Since we know that $D \sim \text{Pois}(d)$, we can evaluate each of the expectations and variances. 
	\begin{align*}
	&(\lambda d)^{2k-2}\cdot p(1-p) \cdot d + \left( 1-p-\frac{1}{q}\right)^2 \cdot (\lambda d)^{2k-2} \cdot d  \\
	=&\  d(\lambda d)^{2k-2} \left( p(1-p) + \left(1-p-\frac{1}{q}\right)^2 \right)
	\end{align*}
	Thus, we have our recursion for $C_k$ as
	\[ C_k = d(1-p)C_{k-1} + dpD_{k-1} + d(\lambda d)^{2k-2}\left( p(1-p) + \left(1-p-\frac{1}{q}\right)^2 \right) \]
	
	Similarly, if we recompute the recursion for $D_k$, we obtain the formula
	\[ D_k = d\cdot \frac{p}{q-1} \cdot C_{k-1} + d\left(1-\frac{p}{q-1}\right) D_{k-1} + d(\lambda d)^{2k-2}\left(\left(\frac{p}{q-1}\right)\left(1-\frac{p}{q-1}\right) + \left(\frac{p}{q-1} - \frac{1}{q}\right)^2 \right) \]
	
	As before, the analysis of this recursion may be difficult, but we solve for an upper bound on both $C_k$ and $D_k$ which will suffice for our purposes. We do this in the same way, computing the recurrence for $F_k = C_k + (q-1)D_k$ . This leads to the same simple recursive form, this time with a different constant term. Rather than $R = 2-p-\frac{p}{q-1}$, we have the term
	\begin{gather*}
	R = \left( p(1-p) + \left(1-p-\frac{1}{q}\right)^2 \right)  + (q-1) \cdot  \left(\left(\frac{p}{q-1}\right)\left(1-\frac{p}{q-1}\right) + \left(\frac{p}{q-1} - \frac{1}{q}\right)^2 \right) 
	\end{gather*}
	The rest of the proof proceeds as in the proof of \cref{varianceYZ} with this revised value of $R$.
\end{proof}

\subsection{Noisy Setting}
\begin{lemma}
	\begin{align*}
	\E{Z_{\rho, k}'} &= \left( \Delta_{11} - \frac{1}{q} \right) \lambda^k d^k  + \frac{d^k}{q} \\
	\Var{Z_{\rho, k}'} &< O(d^k) + \left(\sum_{i=1}^q \Delta_{i1}^2 \right) Rd^k \frac{(\lambda^2 d)^k-1}{\lambda^2 d-1}
	\end{align*}
\end{lemma}

\begin{proof}
	The proof is exactly the same as the proof for \cref{noisyExpVar}. The important point to note is that in the original proof, we conditioned on the variables $Z_{\rho, k}$ and $Y_{\rho, k}^{(i)}$ and used their expectation and variances. However, we never used the fact that they must deterministically add to $d^k$, which is the defining difference between the fixed and random trees. Thus, the same proof holds with the revised value for $R$. 
\end{proof}

By the exact argument as above with the $Y_{\rho, k}^{(i)'}$ instead of $Z_{\rho, k}'$, we get the analogous statements about the distribution of $Y_{\rho, k}^{(i)'}$.

\begin{lemma}
	\begin{align*}
	\E{Y_{\rho, k}^{(i)'}} &=  \left(\Delta_{1i}-\frac{1}{q}\right)(\lambda d)^k + \frac{d^k}{q} \\
	\Var{Y_{\rho, k}^{(i)'}} &< O(d^k) + \left(\sum_{j=1}^2 \Delta_{ji}^2 \right) Rd^k \frac{(\lambda^2 d)^k-1}{\lambda^2 d-1}
	\end{align*}
\end{lemma}

\subsection{Majority Calculation}
\begin{prop}\label{majorityRandom}
	For some constant $C = C(q)$, 
	\[ \liminf_{k \rightarrow \infty}\Prob{M_k} > 1 - \frac{C}{\lambda^2d-1} \]
\end{prop}
\begin{proof}
	The initial majority calculation can now be carried out in the same manner as before in \cref{majority}. The only differences occur in the constant coefficients that appear in the bounds on the variance. In particular, in the non-noisy setting, we only need to check that our new value of $R$ can be bounded by a function of $q$ as before. Recall that 
	\[ R = \left( p(1-p) + \left(1-p-\frac{1}{q}\right)^2 \right)  + (q-1) \cdot  \left(\left(\frac{p}{q-1}\right)\left(1-\frac{p}{q-1}\right) + \left(\frac{p}{q-1} - \frac{1}{q}\right)^2 \right)  \]
	Grouping the first terms together, we have 
	\[ R = p\left( 2 - p -\frac{p}{q-1} \right) + \left( 1-p-\frac{1}{q} \right)^2 + (q-1) \left( \frac{p}{q-1} - \frac{1}{q} \right)^2 \]
	Recall from the previous calculation that $4p\left( 2-p-\frac{p}{q-1} \right) q \leq 4(q-1)$, so we only need to deal with the remaining terms. We have
	\begin{align*}
	\left( 1-p-\frac{1}{q} \right)^2 + (q-1) \left( \frac{p}{q-1} - \frac{1}{q} \right)^2  &= (1-p)^2 - \frac{2(1-p)}{q} + \frac{1}{q^2} + \frac{p^2}{q-1} - \frac{2p}{q} + \frac{q-1}{q^2} \\
	&= (1-p)^2 + \frac{p^2}{q-1} - \frac{1}{q} \\
	&= \frac{(q-1) - 2p(q-1) + p^2(q-1) + p^2}{q-1} - \frac{1}{q} \\
	&= 1 - 2p + \frac{p^2q}{q-1} - \frac{1}{q} \\
	&= 1-p-p\lambda - \frac{1}{q} \\
	&= 1 - \frac{(1-\lambda^2)(q-1)}{q} - \frac{1}{q} \\
	&= \frac{\lambda^2(q-1)}{q}
	\end{align*}
	This expression suffices for our purposes, so we achieve the same result as in the case of the $d$-regular tree. 
\end{proof}

\begin{prop}
	For some constant $C = C(q)$, 
	\[ \liminf_{k \rightarrow \infty}\Prob{\tilde M_k} > 1 - \frac{C}{\lambda^2d-1} \]
\end{prop}
\begin{proof}
	The proof is exactly as  that of \cref{noisyMajority} with the revised value for $R$. Since we showed that the revised value has the desired properties in \cref{majorityRandom}, the same proof suffices to prove this proposition. 
\end{proof}

\section{$\lambda^2d > C\log d$ on Random Trees}\label{toRandom}
In this section, we will generalize the proof of the contraction in the case that $\lambda^2d > C\log d$. Recall that the idea of the proof was to bound the relevant gradient under both "good" conditions which happen with high probability, and also more generally for the "bad" conditions. Naturally, we can extend these conditions to include information about the number of children that the root has. The calculation should then proceed in a similar fashion as before. First, we provide a concentration inequality for the Poisson distribution, as this will inform us about the number of children that the root will have with high probability. From Theorem 1 in \cite{poisson}, we have that

\begin{thmnn}
	Let $X \sim \text{Pois}(d)$ for some $d > 0$. Then for any $x > 0$, we have that
	\[ \Prob{\abs{X - d} \geq  x} \leq 2e^{-\frac{x^2}{2(d+x)}} \]
\end{thmnn}

In particular, taking $x = \frac{d}{2}$, we find that 
\[ \Prob{\abs{X - d} \geq \frac{d}{2}} \leq 2e^{-\frac{d}{12}} \]
So, with high probability, we have that the number of children is in the range $(\frac{d}{2}, \frac{3d}{2})$, and importantly is linear in $d$. We can then carry out the same series of calculations as before, and obtain that in this situation, we have that the gradient is bounded by $(1-\nu)^D$ when $\lambda^2D > C \log D$ where $D$ is the number of children to the root. For $D \in (\frac{d}{2}, \frac{3d}{2})$, we can then obtain a uniform bound of the form $(1-\nu)^d$ whenever $\lambda^2d > C \log d$. With this result in the revised good situation, we can generalize the final calculation. 

\begin{prop}
	Suppose we are on a Galton-Watson Tree where each node independently has $\text{Pois}(d)$ children. Fix $\xi > 0$. Given $\lambda \leq 1-\xi$ and $ q$, there exists $d_0(\lambda, q, \xi)$ such that for all $d \geq d_0$ and all $m \geq m_0(d)$,
	\[ \E{\norm{X_\rho^{(m)} - W_\rho^{(m)}}_1} < \frac{1}{2} \cdot\E{\norm{X_\rho^{(m-1)} - W_\rho^{(m-1)}}_1} \]
	In particular, there exists some constant $C^*(q, \xi)$ so that $\lambda^2d > C^*(q, \xi)\log{d}$ suffices for the above to hold.
\end{prop}
\begin{proof}
	We first condition on the number of children of the root, and expand the expectation as follows. 
	\begin{align*}
	\E{\norm{f(v_1, \ldots, v_D) - f(v_1', \ldots, v_D') } } &\leq \E{\sum_{j=1}^q\sum_{i=1}^D \norm{v_i - v_i'}_1\norm{\frac{\partial f_j}{\partial v_i}}_\infty} \\
	&= \sum_{j=1}^q\E{\sum_{i=1}^D \norm{v_i - v_i'}_1\norm{\frac{\partial f_j}{\partial v_i}}_\infty} \\
	&= \sum_{j=1}^q\E{D \E{ \norm{v_i - v_i'}_1\norm{\frac{\partial f_j}{\partial v_i}}_\infty \middle\vert D }}
	\end{align*}
	
	The question now becomes to estimate the quantity
	\[\E{ \norm{v_i - v_i'}_1\norm{\frac{\partial f_j}{\partial v_i}}_\infty \middle\vert D }\]
	
	We can expand this by further conditioning on the good conditions. Let $G$ be the event that all of $A_i, \tilde A_i$ occur. 
	\begin{gather*}
	\E{ \norm{v_i - v_i'}_1\norm{\frac{\partial f_j}{\partial v_i}}_\infty \middle\vert D \in (\frac{d}{2}, \frac{3d}{2}) }\Prob{d \in (\frac{d}{2}, \frac{3d}{2})}  + \E{ \norm{v_i - v_i'}_1\norm{\frac{\partial f_j}{\partial v_i}}_\infty \middle\vert  D \not \in (\frac{d}{2}, \frac{3d}{2})} \Prob{D \not \in (\frac{d}{2},\frac{3d}{2})}  \\
	\leq \E{ \norm{v_i - v_i'}_1\norm{\frac{\partial f_j}{\partial v_i}}_\infty \middle\vert D \in (\frac{d}{2}, \frac{3d}{2}) } + \frac{q}{\xi} \E{ \norm{v_i - v_i'}_1} \cdot e^{-\frac{d}{12}}
	\end{gather*} 
	
	Next, we consider 
	\[ \E{ \norm{v_i - v_i'}_1\norm{\frac{\partial f_j}{\partial v_i}}_\infty \middle\vert D \in (\frac{d}{2}, \frac{3d}{2}) } \]
	further by conditioning on the good events $G$ as before. This gives the expansion
	
	\[ \E{ \norm{v_i - v_i'}_1\norm{\frac{\partial f_j}{\partial v_i}}_\infty \middle\vert G, D \in (\frac{d}{2}, \frac{3d}{2}) } \Prob{G \middle\vert D \in (\frac{d}{2}, \frac{3d}{2})} \] \[ + \E{ \norm{v_i - v_i'}_1\norm{\frac{\partial f_j}{\partial v_i}}_\infty \middle\vert G^c, D \in (\frac{d}{2}, \frac{3d}{2}) } \Prob{G^c \middle\vert D \in (\frac{d}{2}, \frac{3d}{2})}  \]
	
	Using the bounds in each of the situations, we obtain a bound of the form
	\[ \leq C(q, \xi)(1-\nu)^d \E{\norm{v_i - v_i'}_1} + \frac{q}{\xi} \E{\norm{v_i - v_i'}_1} \cdot \left( \frac{16q+16}{d^2} \right) \]
	
	Putting it all together, we have the general bound that
	\[ \E{ \norm{v_i - v_i'}_1\norm{\frac{\partial f_j}{\partial v_i}}_\infty \middle\vert D } \leq C(q, \xi)(1-\nu)^d \E{\norm{v_i - v_i'}_1} + \frac{q}{\xi} \E{\norm{v_i - v_i'}_1} \cdot \left( \frac{16q+16}{d^2} + e^{-\frac{d}{12}} \right) \]
	
	Notice that this no longer depends on $D$, so we may pull it out of the expectation in the sum, simply leaving $\E{D} = d$. Thus, overall, we get a bound that is 
	\[ qdC(q, \xi)(1-\nu)^d \E{\norm{v_i - v_i'}_1} + \frac{q^2}{\xi} \E{\norm{v_i - v_i'}_1} \cdot \left( \frac{16q+16}{d} + de^{-\frac{d}{12}} \right)  \]
	
	Both of these terms will be small enough once $d$ is large enough as a function of $\lambda$ and $q$, and the same calculation as before along with the fact that $e^{-\frac{d}{12}}$ decreases faster than $\frac{16q+16}{d}$ shows that a condition of the same form $\lambda^2d > C \log d$ as before suffices for this to occur. 
\end{proof}

\section{From $C\log d$ to $C$}
In this section, we relax the condition that $\lambda^2d > C\log d$ to the desired condition that $\lambda^2d > C$. In particular, we will focus on the case where $\lambda \in (\frac{C}{\sqrt{d}}, \frac{C\sqrt{\log d}}{\sqrt{d}})$ is small. We start off with the same telescoping expansion as in the previous cases. We express

\begin{align*}
f_j(X_1, \ldots, X_d) - f_j(\tilde X_1, \ldots, \tilde X_d) = \sum_{k=1}^{d} f_j(X_1, \ldots, X_{k-1}, \tilde X_k, \ldots, \tilde X_d) - f_j(X_1, \ldots, X_{k}, \tilde X_{k+1}, \ldots, \tilde X_d)
\end{align*}

From here, we analyze each of the differences individually, as only one input vector changes in each term. We can Taylor expand in this component to get that

\[f_j(X_1, \ldots, X_{k-1}, \tilde X_k, \ldots, \tilde X_d) - f_j(X_1, \ldots, X_{k}, \tilde X_{k+1}, \ldots, \tilde X_d) = f'(\tilde X_k) (X_k - \tilde X_k) + \frac{1}{2} f^{(2)}(\tilde X_k)(X_k - \tilde X_k)^2 + \cdots \]

Our goal will be to analyze the typical values of the derivatives as well as each of the differences. 

\subsection{Expected Derivative Analysis}
We first handle the analysis of the derivative. Recall the definitions of $N_k$ and $N_{k, -t}$ so that we can express 
\[ \frac{\partial f_j}{\partial X_t(l)} = -N_j \cdot \left( \sum N_k \right)^{-2} \cdot N_{l, -t} \cdot \lambda q \]
when $l \neq j$, taking a similar form when $l = j$. We will manipulate this to a form that will be advantageous to us. We first factor out the term depending on $X_t$ from each of the $N_k$. This gives us 
\[ \frac{\partial f_j}{\partial X_t(l)} = -N_{j, -t} (1+\lambda q(X_t(j) - \frac{1}{q}) \cdot N_{l, -t} \cdot \lambda q \cdot \left( \sum N_{k, -t} (1 + \lambda q(X_t(k) - \frac{1}{q}))\right)^{-2} \]
Since $X_t(k)$ represents a probability, we can sandwich the factored out term by 
\[ 1-\lambda \leq 1+\lambda q(X_t(k) - \frac{1}{q}) \leq 1+\lambda(q-1) \]
Notice that there is exactly one factor of this in the numerator, and one factor from each term in the sum that makes the denominator, so we can factor out these terms and bound this additional factor. 
\[ \frac{1+\lambda q(X_t(j) - \frac{1}{q})}{(1+\lambda q(X_t(k) - \frac{1}{q}))^2} \leq \frac{1+\lambda(q-1)}{(1-\lambda)^2} \leq 4q \ \text{ when $\lambda < \frac{1}{2}$} \]
Since we are in the situation where $\lambda$ is small, this condition can be easily satisfied, so we obtain a bound on the gradient of the form
\[ \abs{\frac{\partial f_j}{\partial X_t(l)}} \leq 4\lambda q^2 \cdot N_{j, -t} \cdot N_{l, -t} \cdot \left( \sum N_{k, -t} \right)^{-2}  \]
This form is particularly important as this expression is now completely independent from the vector $X_t$. This allows us to factor the expectation of the original expression into a product of expectations. 

The next step is to verify that the term $N_{1, -t}$ is much larger than the other terms $N_{j, -t}$, and due to the mixed nature of the inputs $\tilde N_{j, -t}$ as well. Recall that the terms pulled out to obtain $N_{k, -t}$ from $N_{k}$ can be bounded as above, so we can simply work with $N_1$ and $N_j$ along with $\tilde N_j$. 

\begin{prop}
	With probability at least $1-4\exp{-\frac{\lambda^2 d}{8}}$, 
	\[ \frac{N_j}{N_1} \leq e^{-C(q)\lambda^2 d} \]
	The same holds when $N_j$ is replaced with $\tilde N_j$. 
\end{prop}
\begin{proof}
	To simplify, we consider the natural log, so that we work with a sum of terms rather than a product. In particular we write, 
	\[ \log N_j = \sum_{i=1}^d \log (1 + \lambda q(X_i(j) - \frac{1}{q})) = \sum_{i=1}^d \sum_{k=1}^\infty \frac{(-1)^{k+1}}{k}\left( \lambda q (X_i(j) - \frac{1}{q}) \right)^k  \]
	Similarly, we can write out the expansion of the noisy estimates $\tilde N_j$, and compare their differences writing out the first few terms. This gives us the expansion
	\[ \log N_1 - \log  N_j = \sum_{i=1}^d \lambda q(X_i(1) -  X_i(j)) - \frac{1}{2} \lambda^2 q^2 \left( (X_i(1) -\frac{1}{q})^2 - ( X_i(j) - \frac{1}{q})^2 \right) + \cdots \]
	
	We first analyze the expectation and variance of $X_i(1) - \frac{1}{q}$ and $X_i(j) - \frac{1}{q}$. 
	\begin{align*}
	\E{X_i(1) - \frac{1}{q}} &= \E{X_i(1) - \frac{1}{q} \middle \vert \sigma_i = 1} \Prob{\sigma_i = 1} + \E{X_i(1) - \frac{1}{q} \middle \vert \sigma_i \neq 1} \Prob{\sigma_i \neq 1} \\
	&= \E{X^+ - \frac{1}{q}}(1-p) + \E{X^- - \frac{1}{q}}\cdot p 
	\end{align*}
	Similarly, we can compute that 
	\[ \E{X_i(j) - \frac{1}{q}} = \E{X^+ - \frac{1}{q}} \cdot \frac{p}{q-1} + \E{X^- - \frac{1}{q}} \cdot \left(1 - \frac{p}{q-1} \right) \]
	Taking the difference, we find that
	\[ \E{X_i(1) - X_i(j)} = \lambda \E{X^+} - \lambda \E{X^-} \geq \lambda\left( 1 - e^{-C\lambda^2d}\right) \approx \lambda \]
	We can only guarantee that this expectation is close to $\lambda$. However, since we will only use the fact that it is at least a constant multiple, say $\frac{1}{2} \cdot \lambda$, this statement will suffice for the following calculations. 
	To bound the sum in a high probability situation, we use Hoeffding's Inequality in the form that it is presented in Theorem 4 in \cite{Hoeffding}.
	
	\begin{thmnn}
		Let $Z_1, \ldots, Z_n$ be independent bounded random variables with $Z_i \in [a, b]$ for all i, where $-\infty < a \leq b < \infty$. Then
		\[ \Prob{\frac{1}{n} \sum_{i=1}^n (Z_i - \E{Z_i}) \geq  t} \leq \exp(-\frac{2nt^2}{(b-a)^2}) \]
		and 
		\[ \Prob{\frac{1}{n} \sum_{i=1}^n (Z_i - \E{Z_i}) \leq  -t} \leq \exp(-\frac{2nt^2}{(b-a)^2}) \]
		for all $t \geq 0$. 
	\end{thmnn}
	
	Here, we use $Z_i = X_i(1) - X_i(j)$. Notice moreover that $X_i(1) - X_i(j)$ is the difference of two probabilities, and thus is bounded between $-1$ and $1$. Using the theorem, we have that 
	\[ \Prob{\abs{\sum_{i=1}^d (Z_i - \E{Z_i})} \geq dt} \leq 2\exp{-\frac{2dt^2}{4}} \]
	Similarly, we can estimate that the same inequality holds for $Z_i^{(2)}$ where $Z_i^{(2)} = (X_i(1) - \frac{1}{q})^2 - (X_i(j) - \frac{1}{q})^2 $. 
	However, since this summation is multiplied by a factor of $\lambda^2$ from the coefficient, it is of the order $\lambda^3d$ with high probability. Due to the assumptions on $\lambda$ and $d$, we know that this is dominated when $d$ gets large. Thus, we are just left with the first order term, where 
	\[ \lambda q \sum_{i=1}^d (X_i(1) - X_i(j)) \in \lambda^2dq \pm \frac{1}{2} \cdot q\lambda^2 d \text{ w.p. $2\exp{-\frac{\lambda^2d}{8}}$} \]
	For higher order terms, we can use the same inequality to approximate that
	\[ \abs{\frac{1}{a} (\lambda q)^a \sum_{i=1}^d \left((X_i(1) - \frac{1}{q})^a - (X_i(j) - \frac{1}{q})^a\right)} \leq C(q)\lambda^{a+1}d + \frac{\sqrt{a}}{2} \cdot C(q) \lambda^{a+1}d \text{ w.p. $2\exp{-\frac{a\lambda^2 d}{8}}$ }\]
	In particular, when $\lambda^2d$ is at least a large constant, we know that with high probability 
	\[ \log N_j - \log N_1 \leq -C(q) \lambda^2d \]
	when all of these events hold. The probability that at least one event does not hold is bounded by the geometric series 
	\[ \sum_{a=1}^\infty 2\exp{-\frac{a\lambda^2d}{8}} = 2\cdot \frac{e^{-\frac{\lambda^2d}{8}}}{1-e^{-\frac{\lambda^2d}{8}}} < 4e^{-\frac{\lambda^2d}{8}} \]
	The same calculation can be performed with $\tilde N_j$ replacing $N_j$ to obtain a similar result for $N_1$ and $\tilde N_j$. 
\end{proof}
Overall, this gives a bound on the derivative of the form
\[ C(q) \lambda e^{-C'(q)\lambda^2d} \]
in this high probability situation. In the general setting, we know from the Junior Paper that the derivative is always bounded by $\frac{\lambda q}{1-\lambda}$. Since lambda is in the situation we are analyzing, this can be bounded by $C(q)\lambda$. Thus, we know that we can bound the expectation of our derivative by 
\[ C(q) \lambda e^{-C'(q)\lambda^2 d} + 4e^{-\frac{\lambda^2 d}{8}} \cdot C(q)\lambda \]
Both of these terms are linear in $\lambda$ and exponentially small in $\lambda^2d$, which proves the following proposition.
\begin{prop}
	\[ \norm{\E{f'(X_t) \middle\vert \tilde X_t}}_\infty \leq C(q)\lambda e^{-C(q)\lambda^2 d} \]
\end{prop}

\subsection{$X - \tilde X$ Analysis}
Denote $\epsilon_n = \E{X_i^{(n)}(1) - \tilde X_i^{(n)}(1) \middle \vert \sigma_\rho = 1}$. Due to cancellation, we can compute that $\E{X_i(j) - \tilde X_i(j)} $ is on the order of $\lambda \epsilon_n$. 
\begin{lemma}\label{difference}
	For any $i$ and $j$, 
	\[ \E{X_i(j) - \tilde X_i(j)} = C(q) \lambda \epsilon_n \]
\end{lemma}
\begin{proof}
\begin{align*}
\E{X_i(j) - \tilde X_i(j)} &= \sum_{k = 1}^q \E{X_i(j) - \tilde X_i(j) \middle\vert \sigma_i = k} \Prob{\sigma_i = k} \\
&= \sum_{k = 1}^q \E{X_i(j) - \tilde X_i(j) \middle\vert \sigma_i = k} \left(\Prob{\sigma_i = k} - \frac{1}{q} \right) \\
&= \sum_{k \neq j, k=1}^q C(q)\epsilon_n \cdot C_1(q) \lambda + C_2(q)\epsilon_n \cdot C_3(q)\lambda \\
&= C(q) \lambda \epsilon_n
\end{align*}

Here, we can subtract $\frac{1}{q}$ from each probability in the second equality as the expectations they are multiplied by will sum to 0 by symmetry of the communities. 
\end{proof}

The main question is now becomes how to handle the term $(X_i(j) - \tilde X_i(j))^2$, as all higher order terms can be bounded by the first and second order terms. 
\begin{lemma}\label{differenceSquared}
	\[ \E{(X(1) - \tilde X(1))^2} = \Prob{\sigma_\rho = 1} \cdot \epsilon_n \]
\end{lemma}
\begin{proof}
We first write out $X$ and $\tilde X$ conditioning on the noisy leaves on the $n$th sublevel. 
\begin{align*}
\tilde X(1) &= \Prob{\sigma_\rho = 1 \middle\vert \tau^1(n)} = \sum_{\tilde z} \Prob{\sigma_\rho = 1 \middle\vert \tau = \tilde z}  \Prob{\tau = \tilde z \middle\vert \sigma_\rho = 1} \\
X(1) &= \Prob{\sigma_\rho = 1 \middle\vert \sigma^1(n)} = \Prob{\sigma_\rho = 1 \middle\vert \sigma^1(n), \tau^1(n)} \\
&= \sum_{\tilde z} \Prob{\sigma_\rho = 1 \middle\vert \sigma^1(n), \tau(n) = \tilde z} \Prob{\tau(n) = \tilde z \middle\vert \sigma_\rho = 1}
\end{align*}
As these two sums are taken over the same index set, we can evaluate the differences between each of the corresponding terms. In particular, set 
\[ W = \Prob{\sigma_\rho = 1 \middle\vert \tau(n) = \tilde z, \sigma^1(n)} \]
This is a random variable, which has expectation 
\[ \E{W} = \Prob{\sigma_\rho = 1 \middle\vert \tau(n) = \tilde z} \]
Conveniently, these are precisely the terms that appear in the sums for $X$ and $\tilde X$ respectively. Thus, we turn out attention to deriving an equivalent expression for $\E{W - \E{W} \middle\vert \sigma_\rho = 1}$ which provides a relation to $\E{(W - \E{W})^2}$. We evaluate $\E{W \middle\vert \sigma_\rho = 1}$ using a similar technique to Lemma 2.2 in \cite{potts}. 

First, we expand the expectation over all possibilities of $\sigma(n)$, the $n$th sublevel. 
\[ \E{W \middle\vert \sigma_\rho = 1} = \sum_z \Prob{\sigma_\rho = 1 \middle\vert \sigma(n) = z, \tau = \tilde z} \Prob{\sigma(n) = z \middle\vert \sigma_\rho = 1, \tau = \tilde z} \]
Using Bayes' Rule, we can rewrite the second conditional probability within the summation.
\[ = \sum_z \Prob{\sigma_\rho = 1 \middle\vert \sigma(n) = z, \tau(n) = \tilde z} \cdot \frac{\Prob{\sigma_\rho = 1 \middle\vert \sigma(n) = z, \tau(n) = \tilde z}\Prob{\sigma(n) = z \middle\vert \tau(n) = \tilde z}}{\Prob{\sigma_\rho = 1 \middle\vert \tau = \tilde z}} \]
Grouping like terms together, we get the expression
\[ = \sum_z  \Prob{\sigma_\rho = 1 \middle\vert \sigma(n) = z, \tau(n) = \tilde z}^2 \cdot \frac{\Prob{\sigma(n) = z \middle\vert \tau(n) = \tilde z}}{\Prob{\sigma_\rho = 1 \middle\vert \tau(n) = \tilde z}} \]
Rewriting the summation in terms of an expectation over $\sigma(n)$ again, 
\[ = \E{\Prob{\sigma_\rho = 1 \middle\vert \sigma(n), \tau(n) = \tilde z}^2} \cdot \frac{1}{\E{W}} \]
Writing all terms in terms of $W$ and manipulating the expression to contain $\Var{W}$, we get the  relation
\[ \E{W \middle\vert \sigma_\rho = 1} = \frac{\E{W^2}}{\E{W}} = \frac{\E{W}^2 + \Var{W}}{\E{W}}  = \E{W} + \frac{\Var{W}}{\E{W}} \]
Subtracting the expectation to the other side, since it is a constant we get that
\[ \E{W - \E{W}\middle\vert \sigma_\rho = 1} = \frac{\E{(W-\E{W})^2}}{\E{W}} \]
This is exactly the form of expression that we desired. Returning to our original expression, we get that 
\begin{align*}
\E{X(1) - \tilde X(1) \middle\vert \sigma_\rho = 1} &= \sum_{\tilde z} \E{W - \E{W} \middle\vert \sigma_\rho = 1} \Prob{\tau(n) = \tilde z \middle\vert \sigma_\rho = 1} \\
&= \sum_{\tilde z} \E{(W - \E{W})^2} \cdot \frac{\Prob{\tau(n) = \tilde z \middle\vert \sigma_\rho  =1}}{\Prob{\sigma_\rho = 1 \middle\vert \tau(n) = \tilde z}} \\
&= \sum_{\tilde z} \E{(W - \E{W})^2} \cdot \frac{\Prob{\tau(n) = \tilde z}}{\Prob{\sigma_\rho = 1}} \\
&= \frac{1}{\Prob{\sigma_\rho = 1}} \cdot \E{(X(1) - \tilde X(1))^2}
\end{align*}

Using our notation, we get that 
\[ \E{(X(1) - \tilde X(1))^2} = \Prob{\sigma_\rho = 1} \cdot \epsilon_n \]
\end{proof}

\subsection{Putting it Together}
Now that we have the bounds on the gradient and the expected differences, we can bound the expectation of $f(X) - f(\tilde X)$ as a whole. In particular, since $\epsilon_{n+1}$ is defined by the function $f_1$, we focus on 
\[ f_1(X_1, X_2, \ldots, X_d) - f_1(\tilde X_1, \tilde X_2, \ldots, \tilde X_d) \]
and in particular
\[ f_1(X_1, \ldots, X_{k-1}, \tilde X_k, \ldots, \tilde X_d) - f_1(X_1, \ldots, X_{k}, \tilde X_{k+1}, \ldots, \tilde X_d) = f'_1(\tilde X_k) (X_k - \tilde X_k) + \frac{1}{2} f_1^{(2)}(\tilde X_k)(X_k - \tilde X_k)^2 + \cdots \]
We bound each of the terms in this infinite series. First, for the term $\E{f_1'(\tilde X_k)(X_k - \tilde X_k)}$ recall that we bounded the gradient in a way that is independent from $X_k$ and $\tilde X_k$. We want to show that the derivative and $X_k - \tilde X_k$ are approximately independent, and so the expectation will not differ much from the product of the individual expectations, which are
\begin{align*}
\E{f_1'(\tilde X_k)} &\leq C(q)\lambda e^{-C_1(q)\lambda^2d} \\
\E{X_k - \tilde X_k} &\leq C(q) \lambda \epsilon_n
\end{align*}
The following lemma gives the desired behavior
\begin{lemma}
	\[ \E{f'_1(\tilde X_k)(X_k - \tilde X_k)} \leq C(q) \lambda^2 e^{-C(q)\lambda^2d} \epsilon_n \]
\end{lemma}
\begin{proof}
Beginning with expanding the expectation by conditioning on the state of child $k$, we get
\[ \E{f'(\tilde X_k)(X_k(i) - \tilde X_k(i))} = \sum_{j=1}^q \E{f'(\tilde X_k)(X_k(i) - \tilde X_k(i)) \mathbbm{1}(\sigma_k = j)} \Prob{\sigma_k = j} \]
Using the Tower Rule we can condition on the vector $\tilde X_k$ within the expectation. 
\begin{align*}
&= \sum_{j=1}^q \E{\E{f'(\tilde X_k)(X_k(i) - \tilde X_k(i)) \mathbbm{1}(\sigma_k = j)\middle\vert \tilde X_k}} \Prob{\sigma_k = j} \\
\intertext{Since $f'(\tilde X_k)$ and $X_k(i) - \tilde X_k(i)$ are conditionally independent given $\tilde X_k$, we may factor the inner expectation. }
& = \sum_{j=1}^q \E{\E{f'(\tilde X_k)\middle\vert \tilde X_k} \E{(X_k(i) - \tilde X_k(i)) \mathbbm{1}(\sigma_k = j)\middle\vert \tilde X_k}} \Prob{\sigma_k = j} \\
\intertext{As before, we can see that the coefficients of $\Prob(\sigma_k = j)$ sum to 0 by symmetry, so we may subtract $\frac{1}{q}$ from each probability. }
&= \sum_{j=1}^q \E{\E{f'(\tilde X_k)\middle\vert \tilde X_k} \E{(X_k(i) - \tilde X_k(i)) \mathbbm{1}(\sigma_k = j)\middle\vert \tilde X_k}} \left(\Prob{\sigma_k = j} - \frac{1}{q} \right)
\end{align*}
Now, since we have conditioned on $\tilde X_k$, and in particular $\tilde X_k(i)$, as well as the state of child $k$, we know by the Martingale relationship between $X_k(i)$ and $\tilde X_k(i)$ that $\E{(X_k(i) - \tilde X_k(i))\mathbbm{1}(\sigma_k = j) \middle\vert \tilde X_k}$ is always positive or always negative if $i = j$ or $i \neq j$ respectively. As such, we may bound the magnitude of the expectation by the global maximum of the coefficient $\E{f'(\tilde X_k)\middle\vert \tilde X_k}$ and the expectation of the identically positive or negative term. This gives the upper bound 
\[ \abs{\E{f'(\tilde X_k)(X_k(i) - \tilde X_k(i))} } \leq \sum_{j=1}^q\norm{\E{f'(X_k) \middle\vert \tilde X_k}}_\infty \E{\E{(X_k(i) - \tilde X_k(i))\mathbbm{1}(\sigma_k=j) \middle\vert \tilde X_k}} \cdot \abs{\Prob{\sigma_k=j} - \frac{1}{q}} \]
We may now apply the separate bounds on each of these three terms derived previously to get an overall upper bound on the absolute value of 
\[ \leq C(q)\lambda e^{-C(q)\lambda^2d} \cdot \epsilon_n \cdot C(q)\lambda = C(q) \lambda^2e^{-C(q)\lambda^2d}\epsilon_n \]
\end{proof}

For the second order term, we  perform the same calculation on the second derivative to achieve independence. This time, however, we do not need to take advantage of cancellation by subtracting $\frac{1}{q}$ from each probability, and simply use the fact that $(X_k(i) - \tilde X_k(i))^2$ is always positive. Since we have that 
\begin{align*}
\E{f_1^{(2)}(\tilde X_k)} \leq C(q) \lambda^2e^{-C_1(q)\lambda^2d} \qquad \E{(X_k - \tilde X_k)^2} \leq \epsilon_n
\end{align*}
putting these together gives the analogous lemma
\begin{lemma}
\[ \E{f_1^{(2)}(\tilde X_k)(X_k - \tilde X_k)^2} \leq C(q) \lambda^2e^{-C_1(q)\lambda^2d}\epsilon_n \]
\end{lemma}

Finally, for higher order terms, we have a stronger result. We have that the gradient $f_1^{(a)}(\tilde X_k)$ has expectation bounded by $C(q)\lambda^a e^{-C_1(q) \lambda^2d}$. Since $X_k - \tilde X_k \in [-1, 1]$, we know that $\E{\abs{(X_k  - \tilde X_k)^a}} \leq \E{(X_k - \tilde X_k)^2}$, and so is also bounded by $\epsilon_n$. Thus, again by  a similar calculation, we get a bound of the form 
\[ C(q) \lambda^ae^{-C_1(q)\lambda^2d} \epsilon_n  \]
\begin{lemma}
	For all $a \geq 3$, 
	\[ \E{f_1^{(a)}(\tilde X_k)(X_k - \tilde X_k)^a} \leq C(q) \lambda^ae^{-C_1(q)\lambda^2d} \epsilon_n  \]
	in particular it is bounded by a higher power of $\lambda$.
\end{lemma}

 Since $\lambda$ is small in this scenario, these higher order terms will all be dominated by the first and second order. More rigorously, by summing the geometric series we obtain an additional factor of $\frac{1}{1-\lambda}$. As mentioned, $\lambda$ is small, so this term may be comfortably bounded by a constant and absorbed into the already present constant. Thus, overall, we achieve a bound of the form
\[ \E{f_1(X_1, \ldots, X_d) - f_1(\tilde X_1, \ldots, \tilde X_d) \middle\vert \sigma_\rho = 1} \leq \sum_{k=1}^d C(q)\lambda^2e^{-C_1(q)\lambda^2d}\epsilon_n = C(q)\lambda^2d e^{-C_1(q)\lambda^2d} \epsilon_n \]

With this form, there exists a constant $C^*(q)$ such that whenever $\lambda^2d \geq C^*(q)$, we have that $C(q) \lambda^2d e^{-C_1(q) \lambda^2d} < \frac{1}{2}$. With this condition, we achieve the desired contraction. 
\begin{thm}
	There exists some constant $C^*(q)$ so that when $\lambda^2d > C^*(q)$, for all $m \geq m_0(d)$,
	\[ \epsilon_{m+1} < \frac{1}{2} \epsilon_m \]
\end{thm}

\subsection{On Random Trees}
The same argument as in \cref{toRandom} works in this case as well. The key idea to note is that for any value of $D \in (\frac{d}{2}, \frac{3d}{2})$, we can make the same argument and obtain bounds of the same order in terms of $d$ with differing constants. We may then take the maximum of these constants of this bounded interval to obtain the same desired bounds in the good random case. 

\section{$\lambda$ close to 1}
Finally, in this section, we handle the previously excluded case where $\lambda$ is close to 1. Since $\lambda$ is close to 1, we also know that $p$ is close to 0. In particular, we expect that children will match their parents label with high probability, and we can take advantage of this fact to analyze the gradient in a more explicit manner. To do this, we first need to improve on the bound for $E_k$ and $\tilde E_k$ from the majority calculation. Intuitively, we know that as $\lambda$ tends to 1, we should be able to reconstruct the label of the root with more confidence. As such, our bound should depend on a factor of $1-\lambda$, which is currently not the case. To obtain this improved bound, we iterate the majority calculation, and analyze the limiting probability as the depth increases. 

\subsection{Iterated Majority}

\begin{lemma}\label{iteratedMajority}
	There exists a $C(q)$ such that when $\lambda^2d > C(q)$, for any fixed $d \in \mathbb{N}$ there exists $k_0 = k_0(d, \lambda, q)$ so that for all $k \geq k_0$,
	\[  E_k \geq 1 - \frac{2(1-\lambda)(q-1)}{q(0.16d - \lambda)}\]
\end{lemma}

\begin{proof}
	Here, we make the assumption that $p < 0.01$, which can be realized by guaranteeing that $\lambda$ is sufficiently large. Then, in order to ensure that the majority of the child labels are 1, it suffices to have at least $\frac{d}{2}$ children guessed as labeled 1. We first calculate the probability of guessing a child label as 1, regardless of their true label.
	Recall that \[ \liminf_{k \rightarrow \infty} E_k \geq 1 - \frac{C(q)}{\lambda^2d} \]
	In particular, we can find a level $k_0$ so that for all $k > k_0$, 
	\[ E_k \geq 1 - \frac{2C(q)}{\lambda^2d} \]
	Consider now $k \geq k_0+1$. Then, conditioned on the root being 1, we have that each of the $d$ subtrees rooted at its $d$ children are independent, and thus can be correctly reconstructed with probability at least $1-\frac{2C(q)}{\lambda^2d}$. Denote $\hat\sigma$ to be the guess and $\sigma$ to be the true labels. Expanding conditionally,
	\begin{align*}
	p' &= \mathbb{P}(\hat\sigma(u_i) = 1) = \mathbb{P}(\hat\sigma(u_i) = 1 | \sigma(u_i) = 1) \mathbb{P}(\sigma(u_i)=1) + \mathbb{P}(\hat\sigma(u_i) = 1|\sigma(u_i) \neq 1)\mathbb{P}(\sigma(u_i) \neq1) \\
	&= \mathbb{P}(\hat\sigma(u_i) = \sigma(u_i)) (1-p) + \frac{1-\mathbb{P}(\hat\sigma(u_i) = \sigma(u_i))}{q-1}\cdot p \\
	&= E_k(1-p) + \frac{(1-E_k) p}{q-1} \\
	&= \lambda E_k + \frac{1-\lambda}{q}
	\end{align*}
	Assume moreover that $\lambda^2d$ is large enough so that $E_k > 0.99$ as well. With these assumptions, we have safely ensured that the probability of guessing the child as 1 is at least 0.9. We can then use Chebyshev's Inequality to obtain the bound
	\begin{align*}
	E_{k+1} &\geq \Prob{\text{Bin}(d,p' ) \geq \frac{d}{2}} = \Prob{\abs{\text{Bin}(d, p') - dp'} < 0.4d} \\
	&= 1 - \Prob{\abs{\text{Bin}(d, p') - dp'} > 0.4d} \\
	&\geq 1 - \frac{p'(1-p')}{0.16d}
	\end{align*}
	Relaxing this bound slightly, and substituting the expression for $p'$, we know that
	\begin{align*}
	E_{k+1} \geq 1 - \frac{1-p'}{0.16d} = 1 - \frac{1-\lambda E_k - \frac{1-\lambda}{q}}{0.16d}
	\end{align*}
	Note that this gives a recursively increasing lower bound on the probabilities $E_k$. To find the limiting probability, we solve for the stationary point by solving the following equation. 
	\[ E^* = 1- \frac{1-\lambda E^* - \frac{1-\lambda}{q}}{0.16d} \]
	Multiplying both sides by $0.16d$, 
	\[0.16dE^* = 0.16d - 1 + \lambda E^* + \frac{1-\lambda}{q} \]
	Collecting all terms that depend on $E^*$ on the left hand side, 
	\[(0.16d - \lambda) E^* = 0.16d - 1 + \frac{1-\lambda}{q} \]
	Isolating the desired quantity, we can manipulate the expression to the desired form as follows
	\begin{align*}
	E^* &= \frac{0.16d - 1 + \frac{1-\lambda}{q}}{0.16d - \lambda} \\
	&= 1 - \frac{1-\lambda  - \frac{1-\lambda}{q}}{0.16d - \lambda} \\
	&= 1 - \frac{(1-\lambda)(q-1)}{q(0.16d - \lambda)}
	\end{align*} 
	Notice that this makes intuitive sense, as when either $\lambda \to 1$ or $d \to \infty$, the limiting probability tends to 1. Thus, we can now take $k_0$ large enough so that the probability of correct recovery of the root for $k \geq k_0$ is at least 
	\[  E_k \geq 1 - \frac{2(1-\lambda)(q-1)}{q(0.16d - \lambda)}\]
\end{proof}

\subsection{Favorable Inputs Reformulated}
We can now revisit the event of favorable inputs and adjust the lemma for our current situation. 
\begin{lemma}\label{favorableVector2}
	Assume $\lambda^2d > C(q)$ as in Proposition 2.3, then there exists $C = C'(q)$ and $m_0 = m_0(d, \lambda)$ such that the following holds for all $m \geq m_0$. Let $B_i$ be the event that vertex $u_i$ does not satisfy $\max X_{u_i} \geq 1- \frac{1-\lambda}{\lambda q}$, and let $B = \sum_{i=1}^d \mathbbm{1}(B_i)$. Then we have that with probability at least $1- \frac{2}{d^2}$,
	\[ B \leq \frac{2d(q-1)}{q(0.16d - \lambda)} + \sqrt{d \log d} \leq 25 + \sqrt{d \log d} \]
	In particular, with high probability we have that $B = o(d)$. 
\end{lemma}

\begin{proof}
	The proof is analogous to the proof of Lemma 3.2. 
\end{proof}

\subsection{Gradient Analysis Revisited}

\begin{lemma}\label{gradientRevisited}
	Suppose that $\lambda$ is large enough so that $p < 0.01$ and $1-\lambda < \frac{q}{2^7}$. Then we have that for $d$ large enough, 
	\[ \abs{\frac{\partial f_j}{\partial v_t(l)}} \leq 32q \cdot \frac{1}{2^{0.2d-1}} \]
\end{lemma}

\begin{proof}
	As before, with the favorable inputs in place, we consider the form of the gradient. Letting $C$ be the denominator, and $\delta$ be the small terms $N_i$, we have 
	\[ \abs{\frac{\partial f_j}{\partial v_t(l)}} \leq \xi^{-1} \cdot \frac{Cq^2\lambda \delta}{C^2} = \xi^{-1}q^2\lambda\cdot \frac{\delta}{C} \]
	In order to maximize the term $\frac{\delta}{C}$ we consider the bounds defined previously as $\delta_+$ and $C_-$. 
	\begin{align*}
	N_1&> \left(1 + \lambda(q - 1 -\delta q)\right)^{d(1-p) - o(d)}\left(1-\lambda\right)^{dp + o(d)} = C_- \\
	N_i&< \left(1 + \lambda(q - 1)\right)^{dp/(q-1) +  o(d)}\left(1-\lambda(1-\delta q)\right)^{d(1-p/(q-1)) - o(d)} = \delta_+
	\end{align*}
	Recall that we assumed $p < 0.01$. We can then relax $\delta_+$ to be 
	\begin{align*}
	\delta_+ &= (1 + \lambda(q-1))^{0.1d + o(d)} (1-\lambda(1-\delta q))^{0.9d - o(d)} \\
	\intertext{and relax $C_-$ to be}
	C_- &= (1+\lambda(q-1-\delta q))^{0.9d - o(d)} (1-\lambda)^{0.1d + o(d)} \\
	\intertext{Taking $d$ large enough so that the $o(d)$ term is at most $0.3d$, we further obtain the bounds}
	\delta_+ &= (1+\lambda(q-1))^{0.4d}(1-\lambda(1-\delta q))^{0.6d} \\
	\intertext{and}
	 C_- & = (1+\lambda(q-1-\delta q))^{0.6d} (1-\lambda)^{0.4d} 
	\end{align*}
	Simplifying the bases of each of the exponents with the selection that $\delta = \frac{1-\lambda}{\lambda q}$, we have that 
	\begin{align*}
	1+\lambda(q-1) &= q(1-\xi) + \xi \leq q  \\
	1 - \lambda(1-\delta q) &= \xi + \lambda \delta q \leq 2\xi \\
	1 + \lambda(q-1-\delta q) &= q(1-\xi) + \xi - \lambda \delta q \geq \frac{q}{2} \\
	1 - \lambda &= \xi
	\end{align*}
	Putting these bounds together, we find that 
	\begin{align*}
	\frac{\delta}{C} &\leq \frac{q^{0.4d}(2\xi)^{0.6d}}{(q/2)^{0.6d}\xi^{0.4d}} = 2^{1.2d}\left( \frac{\xi}{q}\right)^{0.2d}
	\end{align*}
	Returning to the gradient, we have the bound that 
	\[ \abs{\frac{\partial f_j}{\partial v_t(l)}} \leq \xi^{-1} q^2 \lambda \cdot 2^{1.2d} \left(\frac{\xi}{q}\right)^{0.2d} \leq  q2^{1.2d} \cdot \left( \frac{\xi}{q} \right)^{0.2d - 1} = 32q \left(\frac{2^6 \xi}{q} \right)^{0.2d - 1}  \]
	As long as 
	\[ 2^{1.2} \cdot \left( \frac{\xi}{q} \right) ^{0.2} < 1 \iff \xi < \frac{q}{2^6} \]
	we will have that the bound on the gradient is exponentially small in $d$. 
\end{proof}

Since $q \geq 3$, it suffices for $\lambda \geq 0.95$ for the assumptions of this lemma to be true, which is weaker than the other assumptions made on the size of $\lambda$. Thus, proceeding as in the original proof resolves the case where $\lambda$ close to 1. This gives the updated contraction and main theorem
\begin{prop}\label{contraction2}
	Given $\lambda$ and $ q$, there exists $d_0(\lambda, q)$ such that for all $d \geq d_0$ and all $m \geq m_0(d)$,
	\[ \E{\norm{X_\rho^{(m)} - W_\rho^{(m)}}_1} < \frac{1}{2} \cdot\E{\norm{X_\rho^{(m-1)} - W_\rho^{(m-1)}}_1} \]
	In particular, there exists some constant $C^*(q)$ so that $\lambda^2d > C^*(q)\log{d}$ suffices for the above to hold.
\end{prop}
\begin{thm}\label{equalLimit2}
	Suppose that $\lambda^2d > C^*(q)\log d$  where $C^*$ is as in \cref{contraction2}, then 
	\[ \lim_{m \rightarrow \infty} E_m = \lim_{m \rightarrow \infty} \tilde E_m \]
\end{thm}

\subsection{On Random Trees}
Once again, we claim that the same argument as in \cref{toRandom} suffices in this section. Since we have the same results from the simple majority in the noisy and non-noisy regimes, \cref{iteratedMajority} will hold on the random tree as well. Then, as before, we may condition on $D$, the number of children to the root, and obtain analogous results to \cref{favorableVector2} and \cref{gradientRevisited} when $D$ is linear in $d$. Finally, the same conditional calculation gives the desired \cref{contraction2} and \cref{equalLimit2} on the random tree as well. 

\section{The Algorithm}

\subsection{Black-box Algorithm}\label{blackboxsection}
As the black box algorithm, we will make minor adjustments to the following theorem from Chin, Rao, and Vu. 
\begin{thmnn}\label{blackbox}
	There exists constants $C_1$, $C_2$, such that if $q$ is any constant as $n \rightarrow \infty$ and if
	\begin{enumerate}
		\item $a > b \geq C_1$
		\item $(a-b)^2 \geq C_2q^2a \log \frac{1}{\gamma}$
	\end{enumerate}
	then we can find a $\gamma$-correct partition with probability at least $1-o(1)$ using a simple spectral algorithm. Here, a $\gamma$-correct partition is defined as a collection of subsets $V_1', \ldots, V_q'$ such that $\abs{V_i \cap V_i'} \geq (1-\gamma)\frac{n}{q}$ for all $i$. 
\end{thmnn}

We now describe the alterations we make to this algorithm for it to more suitably fit our needs. First, we would like for the partition to be as uniform as possible. Let $V_1', \ldots, V_q'$ be the output from \cref{blackbox}. Notice that since $\abs{V_i \cap V_i'} \geq (1-\gamma)\frac{n}{q}$, we know that $\abs{V_i'} \geq (1-\gamma)\frac{n}{q}$. Since the total number of vertices is $n$, this provides an upper bound that $\abs{V_i'} \leq (1+\gamma(q-1))\frac{n}{q}$ as well. In particular, to even out the partition, we transfer at most $\gamma(q-1) \frac{n}{q}$ from any of the partition sets. In the worst case, all of these were originally labeled correctly, so we can guarantee that the uniformized partition $V_1^*, \ldots, V_q^*$ is a $q\gamma$-correct partition. 

The second alteration is made to relax the additional condition provided by $a > b \geq C_1$. Suppose that the given block model has parameters $a$ and $b$ that do not satisfy this condition. We may  add edges between pairs of vertices at random, so that the average degree is large enough so that the new graph would have been drawn from the model with parameters $a+C_1$ and $b+C_1$. This guarantees that the first condition is met. With the second condition, we require that 
\[ (a-b)^2 \geq C_2q^2(a+C_1)\log \frac{1}{\gamma} = C_2q^2a\log \frac{1}{\gamma} + C_1C_2q^2\log \frac{1}{\gamma} \]
At this point, we relate this second condition to our parameter $\lambda^2d$ as well. We can calculate using Bayes' Rule that 
\[ \lambda^2d = \frac{(a-b)^2}{q(a + b(q-1))} \]
The denominator is sandwiched between $qa \leq q(a+b(q-1)) \leq q^2a$. With these two inequalities, we can sandwich $\lambda^2d$ by the following
\[ \frac{1}{q^2} \cdot \frac{(a-b)^2}{a} \leq \lambda^2d \leq \frac{1}{q} \cdot \frac{(a-b)^2}{a} \]
In particular, in order for 
\[ (a-b)^2 \geq C_2q^2a\log \frac{1}{\gamma} + C_1C_2q^2\log \frac{1}{\gamma} \]
to hold, it suffices for
\[ \lambda^2d \geq C_2q\log \frac{1}{\gamma} + \frac{C_1C_2q\log \frac{1}{\gamma}}{a} \]
Overall, it will suffice for 
\[ \lambda^2d \geq Cq\log \frac{1}{\gamma} \]
in order for this to hold. With this, we obtain our revised black box algorithm
\begin{thm}\label{blackboxrevised}
	There exists a constant $C$ such that if $q$ is any constant as $n \rightarrow \infty$ and if $\lambda^2d \geq Cq\log (2q^2)$, then we can find a $\frac{1}{2q}$-correct partition such that $\abs{\abs{V_i'} - \abs{V_j'}} \leq 1$ with probability at least $1-o(1)$. 
\end{thm}

 \subsection{Coupling Trees with the Stochastic Block Model}
 In this section, we prove the two coupling lemmas that will allow us to relate all of the previous work on trees to the stochastic block model itself. 
 
 \begin{lemma}
 	Let $R = \lfloor \frac{1}{10 \log(2(a+b))} \log n \rfloor$. For any fixed $v \in G$, there is a coupling between $(G, \sigma')$ and $(T, \sigma)$ such that $(B(v, R), \sigma_{B(v, R)}') = (T_R, \sigma_R)$ a.a.s.
 \end{lemma}

\begin{proof}
	For the proof, we refer the reader to the proof of Proposition 4.2 in \cite{coupling}, and we point out the differences in our setting. 
	
	For a vertex $v \in T$, we define $Y_v$ to be the number of children of $v$, and $Y_v^i$ to be the number of children whose label is $i$. Note that with these definitions, because of Poisson thinning
	\[ Y_v^i \sim \begin{cases}
	\text{Pois}(\frac{a}{q}) & i = \sigma_v \\
	\text{Pois}(\frac{b}{q}) & i \neq \sigma_v
	\end{cases} \]
	and that the pair $(T, \sigma)$ can be entirely reconstructed from $\sigma_\rho$  and the sequences $\{ (Y_u^i)_{u \in T} \}$. 
	
	For $G_R = B(v, R)$, we make similar definitions. Using the notation from \cite{coupling} that $V = V(G)$ and $V_R = V(G) \setminus V(B(v, R))$, we define $\{W^i\}$ to be the partition of $W \subset V$ into vertices with the corresponding label $i$. For any $v \in \partial G_R$, define $X_v$ to be the neighbors of $v$ in $V_R$. In this case, we have that 
	\[ X_v^i \sim \begin{cases}
	\text{Binom}(\abs{V_R^i}, a/n) & i = \sigma_v \\
	\text{Binom}(\abs{V_R^i}, b/n) & i \neq \sigma_v
	\end{cases} \]
	The coupling relies on the fact that the Poisson and Binomial distributions with approximately the same expectations have asymptotically small total variation distance. However, the sequences of variables $X_v^i$ are not sufficient to reconstruct $G_R$. The reasons are because it is possible for two vertices $u, v$ to share children as we are no longer guaranteed to be working on a tree. The Lemmas 4.3 through 4.6 analyze this bad event, and all proofs go through the same as they do not depend on the number of communities present in the graph. 
	
	Armed with these lemmas, we may now consider the conclusion to the proof. Here, the only argument which relies on the two communities is the event $\tilde \Omega$ which is defined as $\abs{\abs{V^+} - \abs{V_-}} \leq n^{3/4}$ and satisfies $\Prob{\tilde\Omega} \rightarrow 1$ exponentially fast. To achieve the same result, we define out event $\tilde \Omega$ to be $\abs{\abs{V^i} - \frac{n}{q}} \leq n^{3/4}$ for all $i$. It remains to show that we still have $\Prob{\tilde \Omega} \rightarrow 1$ exponentially fast. Recall that $\abs{V^i} \sim \text{Binom}(n, \frac{1}{q})$. By Hoeffding's Inequality, we know that
	\[ \Prob{\abs{\abs{V^i} - \frac{n}{q}} \geq n^{3/4}} \leq 2\exp{-2n^{1/2}} \]
	Union bounding over the $q$ communities, we see that indeed $\Prob{\tilde \Omega} \rightarrow 1$ exponentially fast as we wanted. Finally, in order to conclude instead of union bounding over 2 communities we union bound over $q$ communities. In either case, the union bound is over a quantity constant relative to $n$, so the proof concludes with the same result.
\end{proof}
 
 \begin{lemma}
 	For any fixed $v \in G$, there is a coupling between $(G, \tau')$ and $(T, \tau)$ such that $(B(v, R), \tau_{B(v, R)}') = (T_R, \tau_R)$ a.a.s. where $\tau'$ are the labels produced by \cref{blackboxrevised}.
 \end{lemma}

\begin{proof}
	For this proof, we refer to the argument following Lemma 5.9 in \cite{beliefpropogation}. In the same spirit, we condition on $\sigma', B(v, R-1)$ and $G'$ and show that the conditional distribution of $\tau'$ is close to the distribution of $\tau$ conditioned on $\sigma$ and $T$. For any $u \in \partial B(v, R-1)$, we have that 
	\[ \abs{\{ wu \in E(G):  w \in G', \sigma'_w = i, \tau'_w = j \}} \sim \begin{cases}
	\text{Binom}\left( \abs{V^i \cap W_v^j}, \frac{a}{n} \right) & \sigma_u = i \\
	\text{Binom}\left( \abs{V^i \cap W_v^j}, \frac{b}{n} \right) & \sigma_u \neq i
	\end{cases}  \]
	We define $\Delta_{ij} = \frac{q}{n} \abs{V^i \cap W_v^j}$. With this definition we have that $\abs{V^i \cap W_v^j} = \Delta_{ij} \cdot \frac{n}{q} \pm O(n^{1/2})$, so by Lemma 4.6 in \cite{coupling}, the above distributions are at total variation distance at most $O(n^{-1/2})$ from $\text{Pois}(a\Delta_{ij}/q)$ and $\text{Pois}(b\Delta_{ij}/q)$ respectively. Notice moreover that on the noisy tree $T$, we have for $u \in L_{R-1}$,
	\[ \abs{\{ wu \in E(T): \sigma_w = i, \tau_w = j \}} \sim \begin{cases}
	\text{Pois}\left( \frac{a}{q} \Delta_{ij} \right)  & \sigma_u = i \\
	\text{Pois}\left( \frac{b}{q} \Delta_{ij} \right)  & \sigma_u \neq i
	\end{cases}\]
	In particular, the conditional distributions of $\tau$ and $\tau'$ on level $R$ are at total variation distance at most $O(n^{-1/2})$. Union bounding over the $O(n^{1/8})$ choices for $u$, we see that the two distributions are a.a.s the same, which gives the desired coupling. 
	
	We check that the $\Delta$ defined in this way does in fact satisfy the assumptions we specified in the main definitions. For the first, suppose we have that $q \vert n$. Our output from \cref{blackboxrevised} guarantees a uniform partition, so that $\abs{W_v^j} = \frac{n}{q}$ for all $j$. Summing $\Delta_{ij}$ over all $i$ precisely gives $\frac{q}{n} \cdot \abs{W_v^j}$, which in this case is exactly 1. In the case that $q$ does not evenly divide $n$, we will have that $\abs{\abs{W_v^j} - \frac{n}{q}} \leq 1$ but without equality. We instead define 
	\[ \Delta_{ij} = \frac{q}{n} \left( \abs{V^i \cap W_v^j} - \frac{\abs{W_v^j} - n/q}{q} \right) \]
	With this adjustment, we guarantee that the condition holds. Moreover, this new value of $\Delta_{ij}$ differs from our previous definition by at most $O(n^{-1})$. Thus, by the triangle inequality, the Binomial and Poisson distributions are at total variation distance at most $O(n^{-1/2}) + O(n^{-1})$ which is simply the same order as the the $O(n^{-1/2})$ we had previously. For the second assumption, recall that \cref{blackboxrevised} guarantees a $\frac{1}{2q}$-correct partition. This directly translates to the fraction of correct vertices, which is $\Delta_{ii}$ being at least $1 - \frac{1}{q}$. This shows that the definition fits our two assumptions. 
\end{proof}

\subsection{Parameter Estimation}

\begin{lemma}
	Given a set $U \subset V(G)$ such that $\abs{U} = \sqrt{n}$, we have at least $\Theta(n^{1/4})$ vertices with degree $k = \frac{1}{4} \cdot \frac{\log n}{\log \log n}$ a.a.s.
\end{lemma}
\begin{proof}
	First, we note that our degree $k$ satisfies $k^k \sim n^{1/4}$. Indeed, 
	\begin{align*}
		\log k^k &= k \log k = \frac{1}{4} \cdot \frac{\log n}{\log \log n} \cdot ( \log \log n - \log \log \log n - \log 4) \sim \frac{1}{4} \log n
	\end{align*}
	Let $c = \min\{a, b\}$. We also have that
	\begin{align*}
		\Prob{\text{Bin}(n, \frac{c}{n}) = k} &= \binom{n}{k} \left( \frac{c}{n} \right)^k \left(1 - \frac{c}{n}\right)^{n-k} \\
		&\geq \left( \frac{n}{k} \right)^k \left( \frac{c}{n} \right)^k e^{-c} \\
		&= \frac{c^k}{k^k} e^{-c} \\
		&\geq C n^{-1/4c} e^{-c}
	\end{align*}
   If we let $N$ be the number of vertices in $U$ that have degree $k$, then 
   \[ \E{N} \geq \sqrt{n} \cdot Cn^{-1/4c}e^{-c} \geq \Theta(n^{1/4})  \]
   Note that $N$ can be expressed as the sum of indicator variables 
   \[ N = \sum_{u \in U} \mathbbm{1}(A_u) \]
   where $A_u$ is the event that $\deg(u) = k$. In particular, to analyze the distribution of $N$, we would like to understand 
   \[ \Delta = \sum_{u, v \in U} \text{Cov}(A_u, A_v) \]
   These two events are independent given the status of the edge between $u$ and $v$, so we expand the probability accordingly. 
   \begin{align*}
   	\Delta &= \sum_{u, v \in U} \Prob{uv \in E(G)}\Prob{\text{Bin}(n-1, \frac{a}{n}) = k-1)}^2  + \Prob{uv \not \in E(G)}\Prob{\text{Bin}(n-1, \frac{a}{n}) = k)}^2 - \Prob{A_u}^2  \\
   	&\approx \frac{C}{n} \cdot \Prob{\text{Bin}(n-1, \frac{a}{n}) = k-1}^2 - \frac{C}{n} \cdot \Prob{\text{Bin}(n, \frac{a}{n}) = k}^2 \\
   	&\approx \frac{C}{n}
   \end{align*} 
   Since $N$ is the sum of indicators, we have the inequality that
   \[ \Var{N} \leq \E{N} + \sum_{u, v \in U} \text{Cov}(A_u, A_v) \]
   In particular,we have that
   \begin{align*}
   	\Var{N} &\leq \Theta(n^{1/4})  + \binom{\sqrt{n}}{2} \cdot \frac{C}{n} = \Theta(n^{1/4}) + C = o(\E{N}^2)
   \end{align*}
	Thus, by an application of Chebyshev's inequality, we find that $N \sim \E{N}$ a.a.s.
\end{proof}

With this result, we can accurately estimate the noise matrix corresponding with our black box predictions. In particular, notice that with degrees tending to infinity, we can almost surely estimate the label of that corresponding vertex. Moreover, the number of such high degree vertices also tends to infinity, so we can accurately find at least one vertex of each community almost surely. Given a representative of high degree from each community, we can then simply estimate $\Prob{\sigma_v = j \middle\vert \sigma_u = i}$ by looking at the fraction of vertices adjacent to the representative from community $i$ whose label is $j$. This quantity can be written as an expression depending only on $\Delta_{ij}$ using Bayes' Rule, and so we can accurately solve for each entry $\Delta_{ij}$ using this process. 

\subsection{Algorithm}
Finally, we can present the algorithm that produces the optimal partition of vertices. Our algorithm closely follows the same basic structure as the one presented in \cite{beliefpropogation}.

\begin{algorithm}[H]\label{algorithm}
	\SetAlgoLined
	$R \leftarrow \lfloor \frac{1}{10 \log (2(a+b))} \log n \rfloor$ \\
	Take $U \subset V$ to be a random subset of size $\lfloor \sqrt{n} \rfloor$ \\
	$\{ W_*^i \} \leftarrow \emptyset$ \\
	$\{ W_{align}^i \} \leftarrow \text{BBPartition}(G)$ \\
	\For{$v \in V \setminus U$}{
		$\{ W_v^i \} \leftarrow \text{BBPartition}(G \setminus B(v, R-1))$ \\
		Monte Carlo estimate $\Delta$ using the high degree vertices in $U$ \\
		Permute $\{ W_v^i \} $ to align with $\{ W_{align}^i \}$ \\
		Define $\tau' \in [q]^{\partial B(v, R)}$ by $\tau'_u = i$ if $u \in W_v^i$ \\
		Add $v$ to $W_*^{\arg \max (\tilde X_{v, R}(\tau'))}$ \\
	}
	\For{ $v \in U$ } {
		Assign $v$ to one of the $W_*^i$ uniformly at random
	}
	\caption{Optimal Reconstruction Algorithm}
\end{algorithm}

With the proven theorems, to check correctness, we only need to check that we can align the various calls to \cref{blackboxrevised}. Since we are guaranteed at least $\frac{2q-1}{2q} \cdot \frac{n}{q}$ correct vertices in each community, a correct pairing between runs will have at least $\frac{q-1}{q} \cdot \frac{n}{q}$ vertices in common. On the other hand, an incorrect pairing will have at most $\frac{1}{q} \cdot \frac{n}{q}$ vertices in common. Going through each pair one by one, an alignment between calls takes at most $\frac{n(n-1)}{2}$ such comparisons. With this alignment, we can run Algorithm \ref{algorithm}, and the couplings in the previous section show that this algorithm is in fact optimal in the case of $q$ communities in general, which is what we set out to achieve in this paper.

\end{document}